\documentclass[12pt]{article}
\usepackage{mathrsfs}
\usepackage{amsmath}
\usepackage{float}
\usepackage{cite}
\usepackage{esint}
\usepackage[all]{xy}
\usepackage{amssymb}
\usepackage{amsfonts}
\usepackage{amsthm}
\usepackage{color}
\usepackage{graphicx}
\usepackage{hyperref}
\usepackage{bm}
\usepackage{indentfirst}
\usepackage{geometry}
\theoremstyle{plain}
\newtheorem{thm}{Theorem}[section]

\newtheorem{defn}[thm]{Definition}
\newtheorem{prop}[thm]{Proposition}

\newtheorem{cor}[thm]{Corollary}
\newtheorem{lem}[thm]{Lemma}
\newtheorem{cla}[thm]{Claim}
\newtheorem{them}[thm]{Theorem}
\theoremstyle{remark}
\newtheorem{ex}[thm]{Example}
\newtheorem{rmk}[thm]{Remark}

\newcommand{\Z}{\mathbb{Z}}

\newcommand{\R}{\mathbb{R}}

\newcommand{\C}{\mathbb{C}}

\newcommand{\CF}{\mathrm{CF}}
\newcommand{\im}{\mathrm{Im}}


\numberwithin{equation}{section}

\title{Construction of holomorphic quilts in Cartesian product of closed surfaces}
\author{Zuyi Zhang}
\date{ }

\begin{document}

\maketitle

\begin{abstract}
    In this article, we modify the proof of holomorphic quilts from Wehrheim and Woodward in \cite{wehrheim2009floer} to construct a specific type of immersed holomorphic quilt, where the symplectic manifolds are closed surfaces. The application is to compare Lagrangian Floer theory with quilted Lagrangian Floer theory, as they relate through Lagrangian correspondence. A potential example is provided to support Bottman and Wehrheim's conjecture \cite{bottman2018gromov} regarding the isomorphism between Lagrangian Floer homology and quilted Lagrangian Floer homology after twisting by bounding cochains.
\end{abstract}

\section{Introduction}
In this article we alter the proof of holomorphic quilts in Wehrheim and Woodward in \cite{wehrheim2009floer} to construct a certain kind of immersed holomorphic quilts where the symplectic manifolds are closed surfaces.

To be precise, let $(F_1.\omega_1)$ and $(F_2,\omega_2)$ be two closed surfaces equipped with symplectic structures. Suppose $L_i\looparrowright F_i$ and $F\looparrowright (F_1\times F_2,\omega_1\times(-\omega_2))$ are Lagrangian immersions, $i=1,2$. Then the Lagrangian composition $F\circ L_2$ (Definition \ref{def:lacom}) is a Lagrangian immersion in $F_1$ if $F$ and $L_2$ are composable (Definition \ref{def:composable}). Moreover, there is a canonical correspondence between the intersections of $L_1$ with $F\circ L_2$ and $L_1\times L_2$ with $F$ (Proposition \ref{prop:sg}). The main theorem of this paper is the following:

\begin{them}
    Let $(F_1.\omega_1)$ and $(F_2,\omega_2)$ be two symplectic closed surfaces equipped with compatible almost complex structures. Suppose $L_i\looparrowright F_i$ $\mathrm{(}i=1,2\mathrm{)}$ and $F\looparrowright (F_1\times F_2,\omega_1\times(-\omega_2))$ are Lagrangian immersions such that $L_1\times L_2$ intersects $F$ transversely in $F_1\times F_2$. Further more, $F$ and $L_2$ are composable. Assume that
    \[
    u(x,y):\R\times[0,1]\rightarrow F_1
    \]
    is a holomorphic map such that     
    \begin{itemize}
        \item $\lim_{x\rightarrow\pm\infty}u(x,y)=x_{\pm}$, where $x_\pm$ are two points in the intersection of $L_1$ and $F\circ L_2$,
        \item $u$ has its boundary in $L_1$ and $F\circ L_2$ $($Definition \ref{def:lagbd}$)$,
        \item the image of $u$ is contained in the image of $F\looparrowright F_1\times F_2$ after projecting to the first factor.
    \end{itemize} 
    Then there is a holomorphic map $\tilde u(x,y):\R\times[0,1]\rightarrow F_1\times F_2$ such that
    \begin{itemize}
        \item $\lim_{x\rightarrow\pm\infty}\tilde u(x,y)=x_{\pm}$, where $\tilde x_\pm$ are two points in the intersection of $L_1$ and $F\circ L_2$ corresponding to $x_\pm$,
        \item $\tilde u$ has its boundary in $L_1\times L_2$ and $F$.
    \end{itemize}
\end{them}

The aim of the project is to find a combinatorial way to count immersed holomorphic quilts in the case that the symplectic manifolds are closed surfaces. This is a generalization of counting holomorphic strips on closed surfaces. In fact, one can calculate the number of holomorphic strips connecting two intersections of two curves on the surface combinatorially according to Riemann mapping theorem. Therefore on closed surfaces, the calculation of boundary maps in Lagrangian Floer theory is combinatorial. More details can be found in Abouzaid \cite{abouzaid2008fukaya}. On the contrary, the boundary maps of Lagrangian Floer theory in higher dimension is very hard to compute. By Wehrheim and Woodword \cite{wehrheim2010functoriality}, there is a correspondence between holomorphic quilts on a sequence of symplectic manifolds and holomorphic strips on the Cartesian product of all manifold in the sequence. So this project provides a possible way to calculate Lagrangian Floer theory on the Cartesian product of closed surfaces.

It is natural to ask the relation of Lagrangian Floer theory and quilted Lagrangian Floer theory if they are related using a Lagrangian correspondence. Wehrheim and Woodward \cite{wehrheim2010quilted} showed that they are isomorphic with additional assumptions where all the Lagrangians are embedded. In \cite{bottman2018gromov} Bottman and Wehrheim conjectured that by adding the bounding cochains corresponding to the figure-eight bubblings (c.f. \cite{bottman2018gromov}) to twist the boundary maps of the immersed Lagrangian Floer chain group and the immersed quilted Lagrangian Floer chain group, then these two chain groups become chain complexes and are isomorphic. One of the key point to prove the conjecture is to understand how the holomorphic strips and holomorphic quilts are related. This paper gives a partial answer to that.

Another related question is about the Atiyah-Floer conjecture. This conjecture relates the Lagrangian Floer theory and the instanton Floer theory. For the Lagrangain Floer theory part, the symplectic manifold (possibly with singularities) is given by the traceless $\mathrm{SU}(2)$ representations of the fundamental group of a closed surface $F$ modulo conjugations. In fact, Goldman \cite{goldman1984symplectic} proved that this space can be equipped with a symplectic structure. The Lagrangians are the restriction of the the traceless $\mathrm{SU}(2)$ representations of the fundamental group modulo conjugations of a 3-manifold to the boundary, where the boundary of this 3-manifold is $F$. The Lagrangians here are usually immersions into the traceless $\mathrm{SU}(2)$ representations of the fundamental group modulo conjugations of $F$. Examples can be found in Cazassus, Herald, Kirk, and Kotelskiy \cite{cazassus2020correspondence} as well as Herald and Kirk \cite{herald2024endomorphism}. These two papers also explain how holmorophic quilts are related to the Atiyah-Floer conjecture in a more detailed way. At the end of \cite{cazassus2020correspondence}, an example is given to support Bottman and Wehrheim's conjecture \cite{bottman2018gromov}.

The paper is organized as follows. In Section \ref{sec:2}, the necessary background of symplectic manifolds and the definition of holomorhpic quilts are given. Section \ref{sec:3} proves for generic almost complex structures, the moduli space of immsered holomorphic strips with expected dimension 0 is finitely many points. Combining this result with the correspondence between the holomorphic strips and holomorphic quilts, one gets that the moduli space of immsered holomorphic quilts with expected dimension 0 is finitely many points. In Section \ref{sec:4}, we first construct a Hamiltonian perturbation to make all Lagrangians intersect transversely. In Section \ref{sec:6}, the main theorem is proved using the method originated from Wehrheim and Woodward \cite{wehrheim2009floer}. Section \ref{sec:5} provides a potential example that support Bottman and Wehrheim's conjecture in \cite{bottman2018gromov}.\\

\noindent\textbf{Acknowledgement:} The author would like to thank Paul Kirk, Yin Li, Chris Woodward, and Nate Bottman for helpful discussions, as well as Zhangkai Huang for suggestions on improving the paper.

\section{Preliminary}\label{sec:2}
This section begins with some basic facts of symplectic topology. After that, we give the definition of Lagrangian Floer chain groups and the boundary maps. The last part is about the Lagrangian composition and quilted Lagrangian Floer theory.\\

\begin{defn}
Let $X$ be a smooth manifold with $2n$ dimension. A 2-form $\omega$ is called a {\bf symplectic form} if the following conditions hold:
\begin{itemize}
    \item $\omega$ is a closed form.
    \item The restriction of $\omega$ to each tangent space of $X$ is skew-symmetric and non-degenerate.
\end{itemize}
The pair $(X,\omega)$ is used to represent a symplectic manifold $X$ with the symplectic form $\omega$. When no problem is caused, we just say that $X$ is a symplectic manifold without introducing its symplectic form.
\end{defn}

\begin{ex}
The real vector space $\R^{2n}$ is a symplectic manifold with its symplectic form defined as 
\[
\omega_{std}=\sum_{i=1}^ndx_i\wedge dy_i,
\]
where the coordinate of $\R^{2n}$ is given by $(x_1,\ldots,x_n,y_1\ldots,y_n)$.
\end{ex}

\begin{ex}
Any closed surface $F$ admits a symplectic structure. The volume form for any Riemannian metric is a sympletic form.
\end{ex}

\begin{ex}
Suppose that $(X_1,\omega_1)$ and $(X_2,\omega_2)$ are two symplectic manifolds. Then $(X_1\times X_2, w_1\times(-\omega_2))$ is a symplectic manifold.
\end{ex}

\begin{ex}
The cotangent bundle $T^*X$ of any smooth manifold $X$ is a symplectic manifold. 
\end{ex}

\begin{defn}
Let $(X_1,\omega_1)$ and $(X_2,\omega_2)$ be two symmplectic manifolds. Suppose $f:X_1\rightarrow X_2$
is a smooth map.
The function $f$ is a {\bf (local) symplectomorphism} if $f$ is a $($local$)$ diffeomophism and
\[
f^*\omega_2=\omega_1.
\]
In this case, $X_1$ and $X_2$ are {\bf (local) symplectomorphic}.
\end{defn}

\begin{defn}
Let $X$ and $Y$ be smooth manifolds. The triple $(X,f;Y)$ is defined as a {\bf immersed submanifold} of $Y$ with an immersion map
\[
f:X\rightarrow Y,
\]
if $f$ is smooth and its tangent map $df_x$ is injective for each point $x\in X$. For simplicity, denote by $X\looparrowright Y$ an immersion when the immersion map is not specified.
\end{defn}

\begin{defn}
Suppose that $(L,l;X)$ is an immersion from an $n$ dimensional manifold $L$ to a 2n dimensional symplectic manifold $(X,\omega)$. The immersion $(L,l;X)$ is called  a {\bf Lagrangian immersion} if
\[
l^*\omega=0.
\]
\end{defn}

\begin{ex}\label{exp:2.1}
Any immersed curve $C$ in a symplectic surface $(F,\omega)$ is Lagrangian. Since $\omega$ is skew-symmetric, then $\omega(v_x,v_x)=0$ for all $x\in C$ and $v_x\in T_xC$.
\end{ex}

\begin{defn}\label{def:lagint}
Let $(L_1,l_1;X)$ and $(L_2,l_2;X)$ be two Lagrangian immersions into a symplectic manifold $X$. The {\bf fiber product} $L_1\times_X L_2$ is defined as 
\[
L_1\times_X L_2=\{(x_1,x_2)\in L_1\times L_2|\ l_1(x_1)=l_2(x_2)\}.
\]
\end{defn}

\begin{rmk}
The terminology {\bf fiber product} comes from category theory. The intersections of immersed manifolds $(L_1,l_1;X)$ and $(L_2,l_2;X)$ in Definition \ref{def:lagint} cannot be defined simply as $l_1(L_1)\cap l_2(L_2)$. Since $(L_1,l_1;X)$ and $(L_2,l_2;X)$ are immersions, $l_i^{-1}(x)$ may contain more than one element for $x\in l_1(L_1)\cap l_2(L_2)$ for $i=1,2$. It is necessary to distinguish these preimages. So we use fiber products instead of intersections here.
\end{rmk}

The following two theorems are fundamental in symplectic topology .
\begin{them}[Darboux Theorem]{\rm\cite{mcduff2017introduction}}
Let $(X^{2n},\omega)$ be a symplectic manifold. For any point $x\in X$, there is an open neighbourhood $U\subset X$ of $x$ such that $(U,\omega|_U)$ is symplectomorphic to $(\R^{2n},\omega_{std})$.
\end{them}

\begin{them}[Weinstein Tubular Neighbourhood Theorem]\label{thm:tubu}{\rm \cite{eliashberg2002introduction}}
Let $(X,\omega)$ be a symplectic manifold. Assume that $f:L\rightarrow X$ is a Lagrangian immersion. Then there is a local symplectomorphism $G$ from a tubular neighbourhood $T^*_\varepsilon L$ of the zero section of the cotangent bundle $T^*L$ to a neighbourhood $U\supset f(L)$ as in the following commutative diagram:
\begin{equation*}
    \xymatrix{
    T_\varepsilon^*L \ar[dr]^G & \\
    L \ar[r]^{f\ \ \ } \ar[u] &U\subset X}.
\end{equation*}
\end{them}

\begin{rmk}
The local symplectic diffeomorphism $G$ need not to be an injective map.
\end{rmk}

\begin{defn}
Let $(X,\omega)$ be a symplectic manifold. An {\bf almost complex structure} is a bundle map $J:TX\rightarrow TX$ such that 
\[
J^2=-id.
\]
An almost complex structure $J$ is called {\bf $\mathbf\omega$-compatible} if
\begin{itemize}
    \item $\omega(\cdot,J\cdot)$ is a Riemannian metric,
    \item $\omega(J\cdot,J\cdot)=\omega(\cdot,\cdot)$.
\end{itemize}

\end{defn}

\begin{them}{\rm \cite{mcduff2017introduction}}
Let $(X,\omega)$ be a symplectic manifold. Then the space of $\omega$-compatible almost complex structures in nonempty and contractible.
\end{them}

\begin{defn}
Let $(L_1,l^1;X)$ and $(L_2,l^2;X)$ be two Lagrangian immersions into a symplectic manifold $X$. Let $(x_1,x_2)\in L_1\times_XL_2$ be an intersection point of these two Lagrangian immersions. We say {\bf $\mathbf{L_1}$ intersects $\mathbf{L_2}$ transversely at $\mathbf{(x_1,x_2)}$} if 
\[
\mathrm{Im}(df^1)_{x_1}\cap\mathrm{Im}(df^2)_{x_2}=\{0\}.
\]
If the above equation holds for every intersection point, then {\bf $\mathbf{L_1}$ intersects $\mathbf{L_2}$ transversely}.
\end{defn}

\begin{defn}
Let $X$ be a symplectic manifold with compatible almost complex structure $J$. The complex structure $j$ over the strip $\mathbb R\times [0,1]$ is induced from the standard complex structure over $\C$. Then $u:\mathbb R\times [0,1]\rightarrow X$ is called a {\bf J-holomophic map} if
\[
\bar\partial_Ju:=\frac{1}{2}(du+J\circ du\circ j)=0.
\]
\end{defn}

\begin{defn}\label{def:lagbd}
Let $u$ be a smooth map from the strip $\mathbb R\times [0,1]$ to a symplectic manifold $X$ with a compatible almost complex structure. Assume $L_i\looparrowright X$ are 2 immersed Lagrangian submanifolds, for $i=1,2$. We say that {\bf $\mathbf u$ has its boundary in $\mathbf L_i$} if there are lifts $\tilde u_i$ such that the following diagram commutes:
\begin{equation*}
    \xymatrix{
    &L_i\ar[d]\\
    \mathbb R\times \{i-1\}\ar[r]^{\ \ \ u}\ar@{-->}[ur]^{\tilde u_i}&X,
    }
\end{equation*}
for $i=1,2$.
\end{defn}

\begin{defn}
    Let $(L_1,l^1;X)$ and $(L_2,l^2;X)$ be two Lagrangian immersions into a symplectic manifold $X$ such that $L_1$ intersects $L_2$ transversely. Let $J$ be a $\omega$-compatible almost complex structure on $X$. A \textbf{holomorphic bigon connecting two elements} $\mathbf{x_\pm}$ is a $J$-homorphic map $u:\R\times[0,1]\rightarrow X$ satisfying the following conditions:
    \begin{itemize}
        \item $u$ has its boundary in $L_i$, $i=1,2$,
        \item $\lim_{s\rightarrow\pm\infty}u(s,\cdot)=x_\pm$,
        \item the Fredholm index of the Cauchy Riemann operator is 1.
    \end{itemize}
\end{defn}


\begin{defn}
Let $L_1\looparrowright X$ and $L_2\looparrowright X$ be two Lagrangian immersions into a symplectic manifold $X$. Assume that $L_1$ and $L_2$ intersect transversely. The {\bf Lagrangian Floer chain group} $\CF(L_1,L_2;X)$ is defined as the $\Z_2$ vector space generated by points in $L_1\times_XL_2$.
\end{defn}

Before presenting the definition of the boundary maps of Lagrangian Floer chain groups, we review some facts about $J$-holomorphic curves.

Let $X$ be a symplectic manifold with compatible almost complex structure $J$. Let $L_1\looparrowright X$ and $L_2\looparrowright X$ be two Lagrangian immersions. For two intersection points $x_+,x_-\in L_1\times_X L_2$, denote $\mathcal{X}(x_+,x_-)$ as the set of $J$-holomorphic maps $u:\mathbb R\times [0,1]\rightarrow X$ connecting $x_\pm$ with the following properties:
\begin{itemize}
    \item $\bar\partial_Ju=0$,
    \item $\int u^*\omega<+\infty$,
    \item $u$ has its boundary in $L_i$, $i=1,2$,
    \item $\lim_{t\rightarrow\pm\infty}u(t,x)=x_\pm$.
\end{itemize}

There is a translation $\R$ action on $\mathcal{X}(x_+,x_-)$:
\begin{align*}
    \mathcal{X}(x_+,x_-)\times\R&\rightarrow\mathcal{X}(x_+,x_-)\\
    (u(\cdot,\cdot),t)\ \ &\mapsto u(\cdot+t,\cdot).
\end{align*}
The quotient space of $\mathcal{X}(x_+,x_-)$ under this $\R$ action is the moduli space connecting $x_+$ and $x_-$:
\[
\mathcal{M}(x_+,x_-)=\mathcal{X}(x_+,x_-)/\mathbb R.
\]
Suppose $\dim\mathcal{M}(x_+,x_-)=0$, the moduli space $\mathcal{M}(x_+,x_-)$ is a compact smooth manifold for generic $J$-holomorphic structures on $X$ (Section \ref{sec:3}). So there are finitely many elements in $\mathcal{M}(x_+,x_-)$ if $\dim\mathcal{M}(x_+,x_-)=0$. The 
\textbf{boundary map} $\mu^1$ of the Lagrangian Floer complex $\CF(L_1,L_2;X)$ is defined as 
\begin{align*}
\mu^1:\CF(L_1,L_2;X)&\rightarrow \ \ \ \ \ \ \ \ \ \ \ \ \CF(L_1,L_2;X)\\
x_+\ \ \ \ \ \ &\mapsto\sum_{\substack{x_-\in L_1\cap L_2\\ \dim\mathcal{M}(x_+,x_-)=0}}\#_{Mod\, 2}\mathcal{M}(x_+,x_-)x_-,
\end{align*}
where $\#_{Mod\, 2}\mathcal{M}(x_+,x_-)$ is the number of the elements in $\mathcal{M}(x_+,x_-)$ $Mod$ $2$.\\

\begin{ex}\label{exp:abou}\cite{abouzaid2008fukaya}
Let $X$ be a closed surface. Given two immersed curves $L_1\looparrowright X$ and $L_2\looparrowright X$, assume $x_\pm\in L_1\times_XL_2$. Then $\# \mathcal{M}(x_+,x_-)$ is the number of smooth homotopy equivalence class of orientation preserving immersions $u$:
\begin{itemize}
    \item $u:\R\times[0,1]\rightarrow X$,
    \item $u$ has its boundary in $L_i$, $i=1,2$,
    \item $\lim_{t\rightarrow\pm\infty}u(t,x)=x_\pm$,
    \item $x_\pm$ are convex corners of $\im (u)$.
\end{itemize}
If $L_1$ and $L_2$ are properly immersed oriented smooth curves which are the image of a properly embedded curves in the universal cover of $X$, then the map
\begin{align*}
\mu^1:\CF(L_1,L_2;X)&\rightarrow \ \ \ \ \ \ \ \ \ \ \ \ \CF(L_1,L_2;X)\\
x_+\ \ \ \ \ \ &\mapsto\sum_{\substack{x_-\in L_1\times_XL_2\\ \dim\mathcal{M}(x_+,x_-)=0}}\#_{Mod\, 2}\mathcal{M}(x_+,x_-)x_-
\end{align*}
has the property that 
\[
\mu^1\circ\mu^1=0.
\]
For details, see Abouzaid \cite{abouzaid2008fukaya}.
\end{ex}

Example \ref{exp:abou} shows that when the symplectic manifold is two dimensional, the number of $J$-holomorphic maps connecting two intersection points $x_+$ and $x_+$ of Lagrangians immersions is the number of discs connection these two points $x_\pm$.

The following gives the definitions of Lagrangian correspondences and Lagrangian compositions. The concept of Lagrangian correspondences is needed to define Lagrangian compositions. Lagrangian correspondences are special cases of Lagrangian immersions, as in the definition below. When both terms--Lagrangian correspondence and Lagrangian immersion--apply, we use Lagrangian correspondence to stress compositions of Lagrangian immersions.

\begin{defn}
Given symplectic manifolds $(X_0, \omega_0)$ and $(X_1, \omega_1)$, a {\bf Lagrangian correspondence} from $X_0$ to $X_1$ is a Lagrangian immersion
\[
l=(l_0,l_1):L\rightarrow X_0
\times X_1,
\]
where $X_0\times X_1$ is equipped with the symplectic structure $\omega_0\times(-\omega_1)$.
\end{defn}

\begin{defn}\label{def:lacom}
Let $(X_0, \omega_0)$, $(X_1, \omega_1)$, $(X_2, \omega_2)$ be three symplectic manifolds. Then $(X_0\times X_1,\omega_0\times (-\omega_1))$, $(X_1\times X_2,\omega_1\times (-\omega_2))$, $(X_0\times X_2,\omega_0\times (-\omega_2))$ are symplectic manifolds. Assume that 
\[
l^0=(l_0^0,l^0_1):L_{01}\rightarrow X_0
\times X_1\quad and\quad l^1=(l_1^1,l^1_2):L_{12}\rightarrow X_1
\times X_2
\]
are Lagrangian correspondences. The {\bf Lagrangian composition} $L_{01}\circ L_{12}$ of $L_{01}$ and $L_{12}$ is defined as
\begin{equation}\label{equ:lagc}
    l^0\circ l^1:=(l^0_0,l^1_2):L_{01}\circ L_{12}\rightarrow X_0\times X_2,
\end{equation}
where 
\[
L_{01}\circ L_{12}:=L_{01}\times_{X_1} L_{12}=\{(x,y)\in L_{01}\times L_{12}|l^0_1(x)=l^1_1(y)\}.
\]
\end{defn}

\begin{rmk}\label{rmk:lagcom}
As a matter of fact, 
\[
L_{01}\circ L_{12}=(l^0\times l^1)^{-1}(X_0\times\Delta_{X_1}\times X_2),
\]
where 
\[
\Delta_{X_1}:=\{(x,x)|x\in X_1\}
\]
is the diagonal of $X_1\times X_1$. Therefore, if $l^0\times l^1$ is transverse to $X_0\times\Delta_{X_1}\times X_2$, then $L_{01}\circ L_{12}$ is a smooth manifold according to the transversality theorem.
\end{rmk}

The composition $(L_{01}\circ L_{12},l^0\circ l^1;X_0\times X_1)$ in Definition \ref{def:lacom} is not always a Lagrangian immersion. The following proposition gives a sufficient criterion.

\begin{prop}[\cite{wehrheim2010functoriality}]\label{prop:lagco}
Let $(X_0, \omega_0)$, $(X_1, \omega_1)$ and $(X_2, \omega_2)$ be three symplectic manifolds. Then $(X_0\times X_1,\omega_0\times (-\omega_1))$, $(X_1\times X_2,\omega_1\times (-\omega_2))$ and $(X_0\times X_2,\omega_0\times (-\omega_2))$ are symplectic manifolds. Suppose that 
\[
\Delta_{X_1}=\{(x,x)|x\in X_1\}\subset X_1\times X_1
\]
is the diagonal and 
\[
l^0=(l_0^0,l^0_1):L_{01}\rightarrow X_0
\times X_1,l^1=(l_1^1,l^1_2):L_{12}\rightarrow X_1
\times X_2
\]
are Lagrangian correspondences. If $l^0\times l^1$ is transverse to $X_0\times\Delta_{X_1}\times X_2$, then 
\begin{itemize}
    \item $L_{01}\circ L_{12}$ is a smooth manifold,
    \item $(L_{01}\circ L_{12},l^0\circ l^1;X_0\times X_2)$ is a Lagrangian immersion.
\end{itemize}
\end{prop}

\begin{defn}\label{def:composable}
Given Lagrangian correspondences $L_{01}$ and $L_{12}$ in $(X_0
\times X_1,\omega_0\times (-\omega_1))$ and $(X_1
\times X_2,\omega_1\times (-\omega_2))$, resp. These two Lagrangian correspondences $L_{01}$ and $L_{12}$ are said to be {\bf composable} if all the assumptions in Proposition \ref{prop:lagco} are satisfied.
\end{defn}

If $L_1\looparrowright X_1$ is a Lagrangian immersion, then $L_1$ can be regarded as a Lagrangian correspondence from $\{pt\}$ to $X_1$. Let $L_{12}\looparrowright X_1\times X_2$ be a Lagrangian correspondence. If $L_1$ and $L_{12}$ are composable, then $L_1\circ L_{12}\looparrowright X_2$ is a Lagrangian immersion. Similarly, if $L_2$ is a Lagrangian immersion into $X_2$, then $L_2$ is a Lagrangian correspondence from $X_2$ to $\{pt\}$. If $L_2$ is composable with $L_{12}$, then $L_{12}\circ L_2\looparrowright X_1$ is a Lagrangian immersion.\\

We end this section by introducing the quilted Lagrangian Floer theory. Let $g_1: F\rightarrow F_1$ and $g_2:F\rightarrow F_2$ be two covering maps between closed surfaces. Denote $\omega_1$, $\omega_2$, $\omega_1\times(-\omega_2)$ as the symplectic forms on $F_1$, $F_2$, $F_1\times F_2$ respectively. Assume that
\[
(g_1,g_2):F\rightarrow F_1\times F_2
\] 
is a Lagrangian correspondence from $F_1$ to $F_2$. Suppose that $L_1\looparrowright F_1$ and $L_2\looparrowright F_2$ are two immersed curves, composable with $F$. Define the {\bf quilted Lagrangian Floer chain group} $\CF^Q(L_1,F,L_2;F_1,F_2)$ as the abelian group generated by elements in $F\times_{F_1\times F_2}(L_1\times L_2)$ over $\mathbb Z_2$.\\
Let $J_1,J_2,J_1\times (-J_2)$ be the almost complex structures compatible with $\omega_1,\omega_2,\omega_1\times(-\omega_2)$, respectively. Suppose that 
\[
u_i:\R\times[i-1,i]\rightarrow F_i,\ \ i=1,2
\]
are $J$-holomorphic maps such that the following diagram commutes
\begin{equation*}
    \xymatrix{
    &L_1\ar[d]\\
    \mathbb R\times \{0\}\ar[r]^{\ \ \ u_1}\ar@{-->}[ur]^{\exists\tilde u_1}&F_1,
    }\ \ 
    \xymatrix{
    &L_2\ar[d]\\
    \mathbb R\times \{2\}\ar[r]^{\ \ \ u_2}\ar@{-->}[ur]^{\exists\tilde u_2}&F_2,
    }\ \ 
    \xymatrix{
    &F\ar[d]\\
    \mathbb R\times \{1\}\ar[r]^{ (u_1,u_2)}\ar@{-->}[ur]^{\exists(\Bar u_1,\Bar u_2)}&F_1\times F_2.
    }
\end{equation*}
If moreover 
\[
\lim_{x\rightarrow+\infty}(u_1,u_2)(x,y)=x_+,\ \ \lim_{x\rightarrow-\infty}(u_1,u_2)(x,y)=x_-,\ \ x_\pm\in F\times_{F_1\times F_2}(L_1\times L_2),
\]
then $(u_1,u_2)$ is called {\bf quilted holomorphic disc connecting} $\mathbf{x_\pm}$. For fixed $x_\pm\in F\times_{F_1\times F_2}(L_1\times L_2)$, the space of all such $(u_1,u_2)$ connecting $x_\pm$ is denoted as $\mathcal{X}^Q(x_+,x_-)$. There is a $\R$ action on $\mathcal{X}^Q(x_+,x_-)$ defined as
\begin{equation*}
    \begin{aligned}
    \R\times\mathcal{X}^Q(x_+,x_-)&\rightarrow\mathcal{X}^Q(x_+,x_-)\\
    (t,(u_1,u_2)(\cdot,\cdot))&\mapsto(u_1,u_2)(\cdot+t,\cdot).
    \end{aligned}
\end{equation*}
The moduli space of quilted holomorphic stripe connecting $x_\pm$ is defined as
\[
\mathcal{M}^Q(x_+,x_-):=\mathcal{X}^Q(x_+,x_-)/\R.
\]
It is proved in the following section that if the expected dimension of $\mathcal{M}^Q(x_+,x_-)$ is 0, then $\mathcal{M}^Q(x_+,x_-)$ is a compact smooth manifold for generic $J_1$ and $J_2$. The {\bf boundary map} $\mu^1_Q$ of the quilted Lagrangian Floer complex $\CF^{Q}(L_1,F,L_2;F_1,F_2)$ is defined as
\begin{equation*}
    \begin{aligned}
    \mu^1_Q:\CF^{Q}(L_1,F,L_2;F_1,F_2)\rightarrow\CF^{Q}(L_1,F,L_2;F_1,F_2)\\
    x_+\mapsto\sum_{\substack{x_-\in F\times_{F_1\times F_2}(L_1\times L_2)\\ \dim\mathcal{M}^Q(x_+,x_-)=0}}\#_{Mod\, 2}\mathcal{M}^Q(x_+,x_-)x_-.
    \end{aligned}
\end{equation*}

\begin{prop}\label{prop:sg}
Let $L_1$, $L_{12}$, $L_2$ be Lagrangian immersions in $(X_1,\omega_1)$, $(X_1\times X_2,\omega_1\times(-\omega_2))$, $(X_2,\omega_2)$, resp. Suppose $L_1$ and $L_{12}$, $L_{12}$ and $L_2$ are composable. Then the generators of the complexes $\CF(L_1,L_{12}\circ L_2; X_1)$, $\CF(L_1\circ L_{12},L_2;X_2)$, $\CF^{Q}(L_1,F,L_2;F_1,F_2)$ can be identified.
\end{prop}

\begin{proof}
Let 
\[
l^1:L_1\rightarrow X_1, l=(l_1,l_2):L_{12}\rightarrow X_1\times X_2,l^2:L_{2}\rightarrow X_2
\]
be the Lagrangian immersions. According to the definition of pull backs,
\begin{align*}
    L_1\times_{X_1}(L_{12}\circ L_2)
    &=L_1\times_{X_1}\{(y,z)\in L_{12}\times L_2|l_2(y)=l^2(z)\}\\
    &=\{(x,y,z)\in L_1\times L_{12}\times L_2|l^1(x)=l_1(y),l_2(y)=l^2(z)\}
\end{align*}
and
\[
L_{12}\times_{F_1\times F_2}(L_1\times L_2)
    =\{(y,x,z)\in L_{12}\times L_1\times L_{2}|(l_1(y),l_2(y))=(l^1(x),l^2(z))\}.
\]
Therefore
\[
L_{12}\times_{F_1\times F_2}(L_1\times L_2)=L_1\times_{X_1}(L_{12}\circ L_2).
\]
Similarly
\[
L_{12}\times_{F_1\times F_2}(L_1\times L_2)=(L_1\circ L_{12})\times_{X_2} L_2.
\]
The proposition follows.
\end{proof}

\section{Moduli space of holomorphic quilts}\label{sec:3}
We define the moduli space of holomorphic quilts using the method in Wehrheim and Woodward \cite{wehrheim2010quilted}. In \cite{wehrheim2010quilted}, Wehrheim and Woodward prove that embedded holomorphic quilts can be regarded as embedded holomorphic strips. The goal of this section is to generalize their result to the immersed cases and prove that when the virtual dimension for the moduli space is 0, the moduli space consists of finitely many points for an open and dense subsets of the almost complex structures on the symplectic manifolds.

Suppose $(F_1,\omega_1)$ and $(F_2,\omega_2)$ are closed symplectic surfaces with compatible almost complex structures $J_1$ and $J_2$ respectively. Let $L_i$ be Lagrangian immersions into $F_i$ for $i=1,2$ and $F$ is a Lagrangian immersion into the symplectic manifold $(F_1\times F_2,\omega_1\times(-\omega_2))$. Given a quilted holomorphic strip $(u_1,u_2)(s,t):\R\times[0,2]\rightarrow F_1\times F_2$ connection $x_\pm$, where $x_\pm$ are elements in the fiber product of $L_1\times L_2$ and $F$, the corresponding holomorphic strip on $(F_1\times F_2, J_1\times J_2)$ is defined as below:
\begin{align*}
    (w_1,w_2):\R\times [0,1]&\rightarrow \ \ \ \ F_1\times F_2\\
    (s,t)\ \ \ \ \ &\mapsto (u_1(s,t),u_2(s,1-t)).
\end{align*}
This map $(w_1,w_2)$ has its boundaries $\R\times\{0\}$ and $\R\times\{1\}$ in $L_1\times L_2$ and $F$ respectively. Conversely, given a holomorphic strip $(w_1,w_2)(s,t):\R\times [0,1]\rightarrow F_1\times F_2$, the quilted holomorphic strip can be recoverd with the expression $(u_1,u_2)(s,t)=(w_1(s,t),w_2(s,1-t))$.\\

Let $(X,\omega)$ be a symplectic manifold with a compactible almost complex structure $J$. {\bf The next goal} is to show that the moduli space of $J$-holomorphic strips mapped into $(X,\omega)$ whose boundary is in two transveral intersecting Lagrangians is a smooth manifold for generic $J$. Here generic stands for an open dense subset of all $J$-holomorphic structures compatible with $\omega$. Then combining this result with the above identification between $J$-holomorphic quilts with $J$-holomorphic strips, one gets that the moduli space of immersed $J$-holomophic quilts is smooth for generic $J$, with the assumption that the Lagrangian immersions intersect transversely.

Suppose that $L_1$ and $L_2$ are two Lagrangian immersions in $X$ that intersect transversely. Let $u:\R\times[0,1]\rightarrow X$ be an immersed strip with its boundary in $L_1$ and $L_2$ connecting $x_\pm\in L_1\times_X L_2$. Denote $W^{1,p}_{x_\pm}(\R\times[0,1],X)$ as the Sobolev $W^{1,p}$ space of all such immersed strips connecting $x_\pm$, $p\in(2,\infty)$. Given $u\in W^{1,p}_{x_\pm}(\R\times[0,1],X)$, $L^p(\R\times[0,1],\Omega^{0,1}\otimes u^*TX)$ is defined as the Sobolev $L^p$ completion of the sections of the bundle $\Omega^{0,1}\otimes u^*TX\rightarrow \R\times[0,1]$. Here $\Omega^{0,1}$ is the bundle of $(0,1)$-forms over $\R\times[0,1]$. Therefore the Cauchy-Rimenann operator induces the map
\begin{equation}\label{equ:3.1}
    \begin{aligned}
    \mathcal{F}: W^{1,p}_{x_\pm}(\R\times[0,1],X)\times \mathcal{J}&\rightarrow L^p(\R\times[0,1],\Omega^{0,1}\otimes u^*TX)\\
    (u,J)&\mapsto \bar\partial_Ju,
\end{aligned}
\end{equation}
where $\mathcal{J}$ is the space of all $J$-holomorphic structures that are compatible with $\omega$. \\

We claim that the differential of $\bar\partial_J$ is surjective for generic $J\in\mathcal{J}$ following the arugment in Paul Seidel \cite{seidel2008fukaya} Section 9k. In fact, Seidel proved the transversality of the moduli space of $J$-holomorphic curves for generic $J$ and Hamiltonian perturbations. The argument in \cite{seidel2008fukaya} proving the transversality with respect to $J$ is independent of the Hamiltonian perturbations. The proof there uses the same functional as in Equation (\ref{equ:3.1}), except the map $u$ there is an embedding. Seidel showed that the differential of $\mathcal{F}$ is surjective and moreover the differential of $\bar\partial_J$ is also surjective for an open dense subset of $\mathcal{J}$. The condition that $u$ is an embedding is not necessary in the proof. So Seidel's proof actually applies here. Therefore the theorem below is correct.

\begin{them}
    Suppose $(X,\omega)$ is a closed symplectic manifold. Let $L_1$ and $L_2$ be two Lagrangian immersions in $X$. Denote $\mathcal{J}$ as the set of compatible almost complex structures on $X$. Then the differential $D\bar\partial_J$ of the Cauchy-Riemann operator 
    \begin{equation*}
        \bar\partial_J:W^{1,p}_{x_\pm}(\R\times[0,1],X)\rightarrow L^p(\R\times[0,1],\Omega^{0,1}\otimes u^*TX)
    \end{equation*}
    is surjective for an open and dense subsets of $\mathcal{J}$.
\end{them}

\begin{defn}
    Denote $\mathcal{J}_{reg}\subset\mathcal{J}$ as the subset where the differential of the Cauchy-Riemann operator is surjective.
\end{defn}

Since $D\bar\partial_J$ is Fredholm with trivial cokernel for $J\subset\mathcal{J}_{reg}$, if moreover the index of $D\bar\partial_J$ is 0, then space of the solution of the Cauchy-Riemann operator is a smooth manifold with dimension 0. According to Gromov compactness, this space consists of finitely many points. So we have the following corollary.

\begin{cor}
    Suppose $(X,\omega)$ is a closed symplectic manifold. Let $L_1$ and $L_2$ be two Lagrangian immersions in $X$. If $J\in\mathcal{J}_{reg}$ and the index of the differential of the Cauchy-Riemann operator is 0, then the solution space of the Cauchy-Riemann operator consists of finitely many points.
\end{cor}

\section{Transversality of Lagrangian immersions}\label{sec:4}
This section first introduces the definition of local Hamiltonian perturbations. Suppose $(F_1,\omega_1)$ and $(F_2,\omega_2)$ are closed symplectic surfaces. Let $L_i$ be Lagrangian immersions into $F_i$ for $i=1,2$ and $F$ is a Lagrangian immersion into the symplectic manifold $(F_1\times F_2,\omega_1\times(-\omega_2))$. If $L_2$ and $F$ are composable, then their Lagrangian composition $F\circ L_2$ is a Lagrangian immersion into $F_1$. The main theorem in this section is to construct a local Hamiltonian perturbation of $F$ such that the induced perturbation on $F\circ L_2$ recovers a given perturbation on $F\circ L_2$.\\

Let $(X,\omega)$ be a symplectic manifold. Assume that $f:L\rightarrow X$ is a Lagrangian immersion. According to the Weinstein tubular neighbourhood theorem \ref{thm:tubu}, the tubular neighbourhood of $f(L)$ can be identified with a neighbourhood $T^*_\varepsilon L$ of the 0-section $T^*L$. \textbf{A local Hamiltonian perturbation of $\mathbf{L}$} is a Hamiltonian flow $\phi_t$ in $T^*_\varepsilon L$ such that $\phi_0=id$ and the Hamiltonian function of $\phi_t$ supports in $T^*_\varepsilon L$.

\begin{rmk}
    In contrast of Hamiltonian perturbations, local Hamiltonian perturbations do not necessarily preserve Lagrangian Floer homology. According to the main theorem in \cite{zhang2024singularities}, given a Lagrangian immersion $(g_1,g_2): F\rightarrow F_1\times F_2$, the singularities of $g_1$ and $g_2$ are generically fold singularities (Definition \ref{def:foldsin}) with finitely many cusps. After writing out the the Lagrangian composition in local coordinates around the fold singularities, one gets that the Lagrangian immersion derived from the Lagrangian correspondence $F$ has non-transveral self-intersections. Since Hamiltonian perturbations are Hamiltonian flows over the symplectic manifolds, these perturbations must preserve non-transveral self-intersections. Therefore Hamiltonian perturbations are not enough.

\begin{figure}[H]
\centering 
\includegraphics[width=0.6\textwidth]{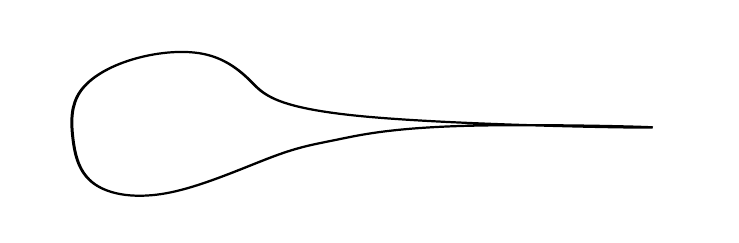}
\caption{The Lagrangian immersion composition has non-transversal self-intersections near the fold singularities without local Hamiltonian perturbation.}
\end{figure}
\end{rmk}

\begin{defn}\label{def:foldsin}
    Suppose $f:X\rightarrow Y$ is a smooth map between two surfaces. A point $x\in X$ is called a \textbf{fold singularity} if there are local coordinates $(x_1,x_2)$ and $(y_1,y_2)$ around $x$ and $f(x)$ respectively such that $f(x_1,x_2)=(y_1,y_2)=(x_1,x_2^2)$.
\end{defn}

\begin{ex}
Suppose the Lagrangian correspondence $F\looparrowright F_1\times F_2$ is given by the following picture. The left and right vertical maps are folds and the top horizontal map is a 0-Dehn twist. The Lagrangian composition of the curve is show in the figure. The following proposition implies that after performing a Hamiltonian perturbation to $F\looparrowright F_1\times F_2$, the Lagrangian composition of the curve is shown in the right of Figure \ref{fig:Lagcomfix}.

\begin{figure}[H]
\centering 
\includegraphics[width=1\textwidth]{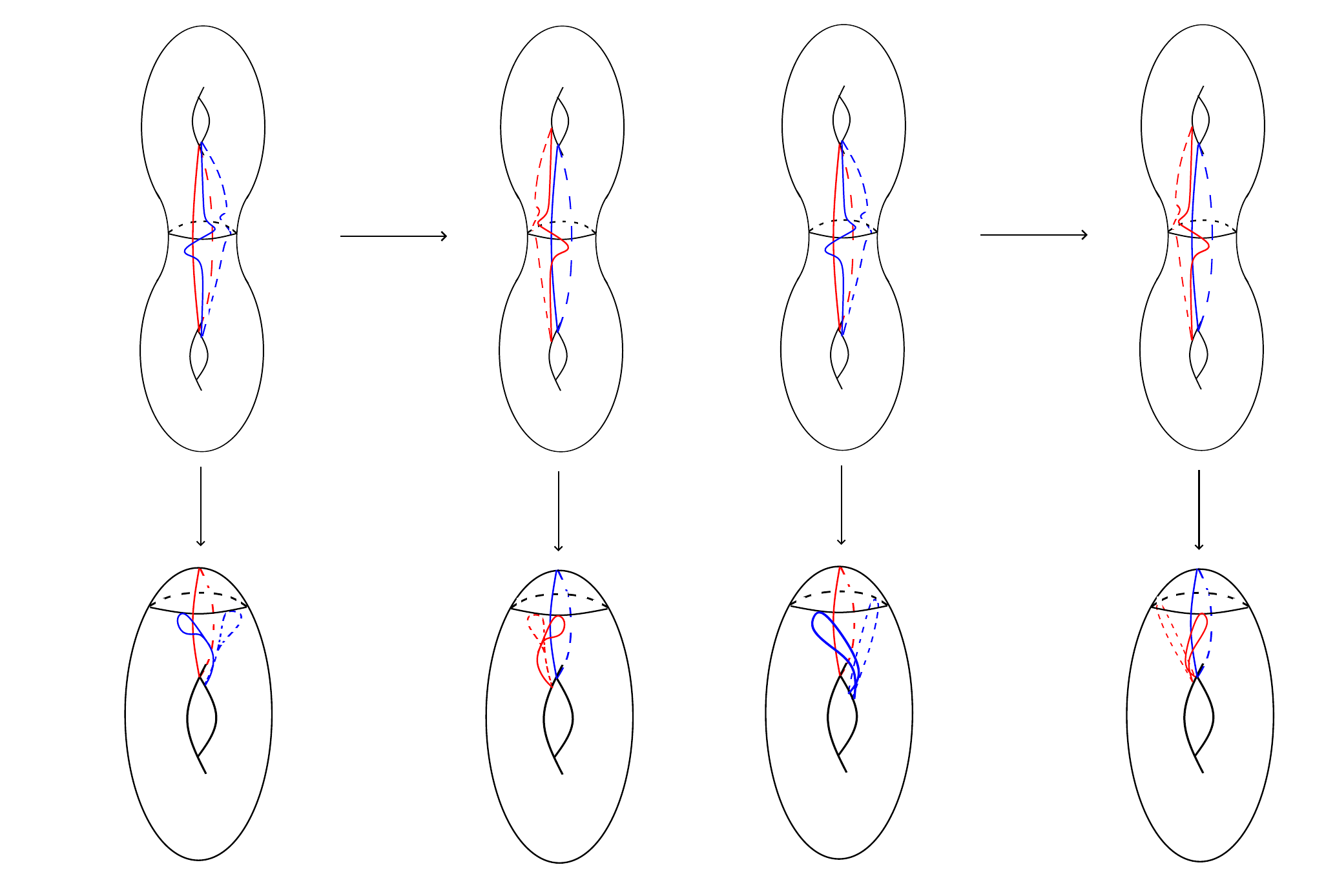}
\caption{The red curve is $L_1$ and the blue one is $L_2$. The picture on the left-hand side is before the local Hamiltonian perturbation, and the picture on the right is after the local Hamiltonian perturbation.}
\label{fig:Lagcomfix}
\end{figure}

\end{ex}

\begin{prop}\label{prop:4.1}
    Suppose $(F_1,\omega_1)$ and $(F_2,\omega_2)$ are closed symplectic surfaces. Let $L_i$ be Lagrangian immersions into $F_i$ for $i=1,2$ and $(g_1,g_2):F\rightarrow(F_1\times F_2,\omega_1\times(-\omega_2))$ be a Lagrangian immersion. Assume the Lagrangian composition $F\circ L_2$ is a Lagrangian immersion into $F_1$. Given a local Hamiltonian perturbation $\phi_t$ of $F\circ L_2$ such that $\phi_t$ fixes a neighbourhood of the image of the singularity of $g_1$, then there is a local Hamiltonian perturbation of $F$ such that the induced perturbation on $F\circ L_2$ that recovers the perturbation $\phi_t$.
\end{prop}

\begin{lem}
    Let $(X,\omega)$ be a symplectic manifold. Assume that $f:L\rightarrow X$ is a Lagrangian immersion. Suppose that $\phi_t$ is a local Hamiltonian perturbation of $L$. Then $\phi_t$ is a composition of several local Hamiltonian perturbations such that the each support of these local Hamiltonian perturbation is arbitrarily small.
\end{lem}

\begin{proof}
    By the definition of $\phi_t$, the Hamiltonian function is supported in a neighbourhood $T^*_\varepsilon L$ of the 0-section in $T^*L$. Denote the Hamiltonian function of $\phi_t$ as $H$. Then $H$ is supported in $T^*_\varepsilon L$. Suppose $\{U_\alpha\}$ is a finite collection of open sets covering the support of $H$ with the corresponding partition of unity $\{\rho_\alpha\}$. Then $H=\sum_\alpha \rho_\alpha H$. As a result, $\phi_t$ is the composition of the Hamiltonian flows induced by $\rho_\alpha H$.
\end{proof}

\begin{proof}[Proof of Proposition \ref{prop:4.1}]
    According to the previous lemma, we can assume the support of $\phi_t$ is small enough. The idea is to identify the support of $\phi_t$ as a subset of $T^*_\varepsilon F$ and extend the Hamiltonian function $H$ of $\phi_t$ to a neighbourhood that is constant in the $F_2$ direction.\\

    Suppose that there is a small enough interval $I$ in the domain of the Lagrangian immersion $l_2^F:F\circ L_2\looparrowright F_1$ away from the singularity of $g_1$. Then
    \[
    l_2^F|_{I}:I\rightarrow F_1
    \]
    is a Lagrangian submanifold. The Weinstein tubular neighbourhood theorem introduces a symplectomorphism $G$ as indicated in the following:
    \begin{equation*}
    \xymatrix{
    T_\varepsilon^*I \ar[dr]^G & \\
    I \ar[r]^{l_2^F|_{I}\ \ \ } \ar[u] &U\subset F_1.}
    \end{equation*}
    Without loss of generality, assume the support of $H$ is contained in $T_\varepsilon^*I$ and the connected component of $g_1^{-1}(U)$ intersecting $I$ non-trivially is differomorphic to $U$. Combining this with the local symplectic diffeomorphism $G^F$ from the Weinstein tubular neighbourhood theorem as in the following diagram, $T_\varepsilon^*I$ can be identified with a 2-dimensional subset $V$ (open in the induced topology) contained in the 0-section of $T_\varepsilon^*F$.
    \begin{equation*}
        \xymatrix{
    T_\varepsilon^*F \ar[dr]^{G^F} & \\
    F \ar[r] \ar[u] &F_1\times F_2}.
    \end{equation*}
    By identifying $T_\varepsilon^*F|_U$ with its image using $G^F$, recalling that $F\looparrowright F_1\times F_2$ is a Lagrangian immersion, then $V$ intersects trasversely with $\{p\}\times F_2,\forall p\in F_1$.

    Regarding $H$ as a function defined on $V$, because $V$ intersects trasversely with $\{p\}\times F_2,\forall p\in F_1$, $H$ can be extended constantly in the $F_2$ direction to get $H_1$. Let $\rho$ be a smooth cut off function defined on $F_1\times F_2$ such that:
    \begin{itemize}
        \item $\rho$ is supported in a small neighbourhood of $V$,
        \item $\rho=1$ on the support of $H$.
    \end{itemize}
    Then $H_2:=\rho H_1$ is a smooth function defined on $T_\varepsilon^*F$ such that $H_2|_V=H$. Therefore the Hamiltonian flow of $H_2$ gives a local Hamiltonian perturbation of $F$ and this local Hamiltonian perturbation induces a perturbation on $F\circ L_2$ that recovers $\phi_t$.
\end{proof}


\section{Construction of quilted holomorphic discs}\label{sec:6}

The goal of this section is to prove the main theorem of this paper. For reader's convenience, the main theorem is restated as follows:

\begin{them}
    Let $(F_1.\omega_1)$ and $(F_2,\omega_2)$ be two symplectic closed surfaces equipped with compatible almost complex structures. Suppose $L_i\looparrowright F_i$ $\mathrm{(}i=1,2\mathrm{)}$ and $F\looparrowright (F_1\times F_2,\omega_1\times(-\omega_2))$ are Lagrangian immersions such that $L_1\times L_2$ intersects $F$ transversely in $F_1\times F_2$. Further more, $F$ and $L_2$ are composable. Assume that
    \[
    u(x,y):\R\times[0,1]\rightarrow F_1
    \]
    is a holomorphic map such that     
    \begin{itemize}
        \item $\lim_{x\rightarrow\pm\infty}u(x,y)=x_{\pm}$, where $x_\pm$ are two points in the intersection of $L_1$ and $F\circ L_2$,
        \item $u$ has its boundary in $L_1$ and $F\circ L_2$ $($Definition \ref{def:lagbd}$)$,
        \item the image of $u$ is contained in the image of $F\looparrowright F_1\times F_2$ after projecting to the first factor.
    \end{itemize} 
    Then there is a holomorphic map $\tilde u(x,y):\R\times[0,1]\rightarrow F_1\times F_2$ such that
    \begin{itemize}
        \item $\lim_{x\rightarrow\pm\infty}\tilde u(x,y)=\tilde x_{\pm}$, where $\tilde x_\pm$ are two points in the intersection of $F$ and $L_1\times L_2$ corresponding to $x_\pm$ (in the sense of Proposition \ref{prop:sg}),
        \item $u$ has its boundary in $L_1\times L_2$ and $F$.
    \end{itemize}
\end{them}

\begin{proof}
The proof is adopted from Wehrheim and Woodward \cite{wehrheim2009floer}. In \cite{wehrheim2009floer}, the the moduli spaces are builted on Lagrangian embeddings. However, the map they construct between the moduli spaces can be generalized to Lagrangian immersions in dimension 2.

Recall that the generators of $\CF(L_1\circ F,L_2;F_2)$ and $\CF^Q(L_1,F,L_2;F_1,F_2)$ can be identified canonically by Proposition \ref{prop:sg}, so we do not distinguish them. \textbf{The first goal} is to construct a map between the moduli spaces of $\CF(L_1\circ F,L_2;F_2)$ and $\CF^Q(L_1,F,L_2;F_1,F_2)$:
\[
\mathcal{M}(x_+,x_-)\rightarrow \mathcal{M}^Q(x_+,x_-),
\]
where $x_\pm$ are the generators of $\CF(L_1\circ F,L_2;F_2)$ and $\CF^Q(L_1,F,L_2;F_1,F_2)$.\\

\textbf{To achieve this goal}, it is necessary to alter the space from $\mathcal{X}^Q(x_+,x_-)$ to $\tilde{\mathcal{X}}^Q(x_+,x_-)$ (Definition \ref{def:quiltbdd}). Denote $\mathcal{X}^Q_\delta(x_+,x_-)$ as the quilted holomorphic stripe whose top strip has width $\delta$. In this notation, the observation is that $\mathcal{X}^Q(x_+,x_-)$ can be identified with $\mathcal{X}^Q_\delta(x_+,x_-)$ for $\delta\in(0,1]$ and
\[
\mathcal{X}^Q(x_+,x_-)=\mathcal{X}^Q_1(x_+,x_-).
\]
\begin{defn}\label{def:quiltbdd}
    Let $\tilde{\mathcal{X}}^Q(x_+,x_-)$ be the space of tuples of $J$-holomorphic maps
    \[
    (u_2,u_2',u_1,u_1')\rightarrow F_2\times F_2\times F_1\times F_1
    \]
    with the following boundary conditions:
    \begin{itemize}
        \item $(u_2,u_2')|_{t=0}$ and $(u_1,u_1')|_{t=0}$ have their boundaries in the diagonals $\Delta_{F_2}$ and $\Delta_{F_1}$, respectively.
        \item $u_2|_{t=1}$ has its boundary in $L_2$.
        \item $((u_2',u_1,u_1')|_{t=\delta})$ has its boundary in $(L_{1}\times F)^T$, where
        \[
        (L_{1}\times F)^T:=\{(x_2,x_1,x_1')\in F_2\times F_1\times F_1|(x_1,x_2,x_1')\in L_1\times F\}.
        \]
    \end{itemize}
\end{defn}

The spaces $\mathcal{X}^Q_\delta(x_+,x_-)$ and $\tilde{\mathcal{X}}^Q_{\bar{\delta}}(x_+,x_-)$ can be identified as follows for $\bar\delta=\frac{\delta}{2-\delta}$. Suppose $u=(u_2.u_2',u_1,u_1')\in \tilde{\mathcal{X}}^Q_{\bar{\delta}}(x_+,x_-)$, define
\begin{equation}\label{equ:strcor1}
v_1(s,t)=\left\{
\begin{aligned}
    &v_1((1+\bar\delta)s,\bar\delta-(1+\bar\delta)t)\quad 0\le t\le\frac12\delta,\\
    &v_1'((1+\bar\delta)s,(1+\bar\delta)t-\bar\delta)\quad \frac12\delta\le t\le\delta,
\end{aligned}
     \right.
\end{equation}
\begin{equation}\label{equ:strcor2}
v_2(s,t)=\left\{
\begin{aligned}
    &v_2'((1+\bar\delta)s,\bar\delta-(1+\bar\delta)t)\quad 0\le t\le\frac12\delta,\\
    &v_2((1+\bar\delta)s,(1+\bar\delta)t-\bar\delta)\quad \frac12\delta\le t\le1.
\end{aligned}
     \right.
\end{equation}   
Then $v=(v_1,v_2)\in \mathcal{X}^Q_\delta(x_+,x_-)$.

\textbf{The idea of constructing elements from $\mathcal{X}(x_+,x_-)$ to $\tilde{\mathcal{X}}^Q_{\delta}(x_+,x_-)$ is described below}. Suppose $u_2(s,t)\in \mathcal{X}(x_+,x_-)$, 
\[
u_2':\R\times [0,\delta]\rightarrow F_2
\]
is defined as the constant extension of $u_2|_{t=0}$ along $[0,\delta]$. Notice that $u_2|_{t=0}$ has its boundary in $L_1\circ F$, there is a lift of $u_2|_{t=0}$ to 
\[
(u_2|_{t=0},\bar{u}_1,\bar{u}_1):\R\rightarrow (L_1\times F)^T.
\]
Set 
\[
u_1=u_1':\R\times [0,\delta]\rightarrow (L_1\times F)^T
\]
as the constant extension of $\bar u_1$ along $[0,\delta]$. Then $u^\delta=(u_2,u_2',u_1,u_1')$ has the same boundary condition as elements in $\tilde{\mathcal{X}}^Q_{\delta}(x_+,x_-)$. We call $u^\delta$ as the {\bf constant extension of $\mathbf u_2$}. To make $u^\delta$ $J$-holomorphic, \textbf{we need to construct some exponential map along some vector field such that the composition of the exponential map with $u$ preserves the Lagrangian boundary condition and the resulting map is $J$-holomorphic}.

Fix $u_2\in \mathcal{X}(x_+,x_-)$, let $u^\delta=(u_2,u_2',u_1,u_1')$ be the constant extension of $u_2$. Consider a section $\eta=(\eta_2,\eta_2',\eta_1,\eta_1')$ of the pullback bundle $(u^\delta)^*T(F_2\times F_1\times F_1)$ with the following boundary conditions:
\begin{equation}\label{equ:bddcon}
    \begin{aligned}
        (\eta_2,\eta_2')|_{t=0}\in \Gamma((u_2,u_2')|_{t+0}^*T\Delta_{F_2}),\quad (\eta_1,\eta_1')|_{t=0}\in \Gamma((u_1,u_1')|_{t=0}^*T\Delta_{F_1})\\
        (\eta_2',\eta_1,\eta_1')|_{t=\delta}\in \Gamma((u_2',u_1,u_1')|_{t=\delta}^*T(L_1\times F)^T), \quad \eta_2|_{t=1}\in \Gamma((u_2)|_{t=1}^*TL_2).
        \end{aligned}
\end{equation}

\textbf{There are two problems to construct an exponential map along $\eta$}. The first problem is that the exponential map is not defined on vector fields over immersed surfaces. The second problem is that exponential map does not preserve the Lagrangian boundary conditions.

\textbf{To solve the first problem}, we can define exponential map locally then the global exponential map can be glued from the local ones by the uniqueness of the exponential maps. Let $U$ be a small enough open neighbourhood of the base manifold of the bundle $u^*T(F_2\times F_2\times F_1\times F_1)$ such that $u^*T(F_2\times F_2\times F_1\times F_1)|_U$ is isomorphic to $T(F_2\times F_2\times F_1\times F_1)|_{u(U)}$. Then $\eta|_{U}$ can be regarded as a section of $T(F_2\times F_2\times F_1\times F_1)|_{u(U)}$. Denote the push forward vector field as $\eta|_{u(U)}$. Then the exponential map over $u(U)$ along $\eta|_{u(U)}$ is well defined. Therefore we have a well defined map $\exp_{u(U)}(\eta|_{u(U)})$ for each small enough $U$. Since the exponential map is unique, these locally defined exponential maps can be glued together to get a global immersion, denoted as $\exp_{u}(\eta)$.

\textbf{However, $\exp_{u}(\eta)$ does not satisfy the Lagrangian boundary condition in Definition \ref{def:quiltbdd}. To fix this}, the following lemma is needed.

\begin{lem}
    Given $u_2\in \mathcal{X}(x_+,x_-)$, let $u^\delta=(u_2,u_2',u_1,u_1')$ be the constant extension of $u_2$. There exist $\epsilon>0$ and smooth families of maps
    \begin{align*}
        Q_s:(u_2',u_1,u_1')^*(s,0)T(F_2\times F_1\times F_1)\supset B_\epsilon&\rightarrow(u_2',u_1,u_1')^*(s,0)T(F_2\times F_1\times F_1)\quad \forall s\in\R,\\
        Q^2_{s,t}:u_2^*(s,t)TF_2\supset B_\epsilon^2&\rightarrow u_2^*(s,t)TF_2\quad \forall (s,t)\in\R\times[0,1],
    \end{align*}
that are diffeomorphisms onto their image, where $B_\epsilon$ and $B^2_\epsilon$ are $\epsilon$ neighbourhoods of the origin in the tangent space at a given point. Moreover, $Q_s$ and $Q^2_{s,t}$ have the following properties:
\begin{itemize}
    \item $Q_s(0)=0$, $dQ_s(0)=0$,$Q_{s,t}^2(0)=0$, $dQ_{s,t}^2(0)=0$. In particular, there is a constant $C_Q$ such that for all $(\eta_2',\eta_1,\eta_1')\in B_\epsilon$ and $\eta_2\in B_\epsilon^2$
    \[
    |Q_s(\eta_2',\eta_1,\eta_1')|\le C_Q|(\eta_2',\eta_1,\eta_1')|^2,\quad |Q_{s,t}^2(\eta_2)|\le C_Q|\eta_2|^2.
    \]
    \item $\exp_{(u_2',u_1,u_1')(s,0)}\circ(1+Q_s)$ maps $(u_2',u_1,u_1')^*(s,0)T(L_1\times F)^T\cap B_\epsilon$ to $(L_1\times F)^T$.
    \item $\exp_{(u_2',u_1,u_1')(s,0)}\circ(1+Q_s)$ maps $(u_2',u_1,u_1')^*(s,0)T(F_2\times\Delta_{F_1})\cap B_\epsilon$ to $F_2\times\Delta_{F_1}$.
    \item $\exp_{u_2(s,1)}\circ(1+Q^2_{s,1})$ maps $u_2^*(s,1)TL_2\cap B_\epsilon^2$ to $L_2$.
    \item Restricting $Q_s$ to $(u_2',u_1,u_1')^*(s,0)T(F_2\times\Delta_{F_1})$ and composing it with the projection $Pr_2:(u_2',u_1,u_1')^*(s,0)T(F_2\times F_1\times F_1)\rightarrow u_2'^*TF_2$ yields a map that is independent of the $(u_1,u_1')^*(s,0)T\Delta_{F_1}$-component. The resulting family
    \[
    Q^2_s:u_2^*(s,0)TF_2\supset B_\epsilon^2\rightarrow u_2^*(s,0)TF_2
    \]
    coincides with $Q^2_{s,0}$.
\end{itemize}
\end{lem}

Notice that these maps are constructed fiberwisely in \cite{wehrheim2009floer}. So the same construction applies when we pulling back everything to the stripes using Weinstein's tubular neighbourhood theorem.

\begin{proof}
    We first construct the map $Q_s$. This map is constructed at every point $s$. First assume $s$ is fixed. Let $\exp_{(u_2',u_1,u_1')(s)}$ be the exponential map defined at $T_{(u_2',u_1,u_1')(s)}F_2\times F_1\times F_1$. Consider the restriction of $\exp_{(u_2',u_1,u_1')(s)}$ to a geodesic ball around the origin. Define
    \[
    \mathcal{L}:=\exp^{-1}_{(u_2',u_1,u_1')(s)}(L_1\times F)^T\ \  \mathrm{(Definition}\ \ref{def:quiltbdd}\mathrm{)}.
    \]
    This is a subset in $X:=T_{(u_2',u_1,u_1')(s)}F_2\times F_1\times F_1$. By assumption, $(L_1\times F)^T$ intersects $\Delta:=T_{(u_2',u_1,u_1')(s)}F_2\times \Delta_{F_1}$ transversely. Then one can assume that the geodesic ball is chosen small enough such that $\mathcal{L}$ intersects $\Delta$ transversely. So this intersection set $\hat{\mathcal{L}}_{12}$ is a smooth manifold and is diffeomorphic to $\mathcal{L}_{12}:=\exp^{-1}_{u_2'(s)}(L_1\circ F)\subset T_{u_2'(s)}F_2$. Thus 
    \[
    \dim{\mathcal{L}}_{12}=\dim \hat{\mathcal{L}}_{12}=1.
    \]
    If $\dim(\mathcal{L}\cap (\{0\}\times T_{(u_1,u_1')(s)}\Delta_{F_1}))\ne0$, then $L_1$ and $F$ are not composable, which is a contradiction. Therefore the following is true.
    \[
    \Delta=T_{u_2'(s)}F_2\times T_{(u_1,u_1')(s)}\Delta_{F_1}=T_{0}\hat{\mathcal{L}}_{12}\oplus((T_{0}\mathcal{L}_{12})^\perp\times\{0\})\oplus(\{0\}\times T_{(u_1,u_1')(s)}\Delta_{F_1}).
    \]
    Moreover, $\hat{\mathcal{L}}_{12}$ can be written as the graph of a map $\psi$ over a sufficiently small $\epsilon$-ball,
    \[
    \psi=(\psi_2^\perp,\psi_{11}):T_{0}\hat{\mathcal{L}}_{12}\supset B_\epsilon\rightarrow(T_{0}\mathcal{L}_{12})^\perp\times(T_{(u_1,u_1')(s)}\Delta_{F_1})
    \]
    with $\psi(0)=0$ and $d\psi(0)=0$.
    Let $C$ be a complement of $T_{0}\hat{\mathcal{L}}_{12}\subset T_{0}\mathcal{L}$, thus
    \[
    X:=T_{0}\mathcal{L}+\Delta=C\oplus T_{0}\hat{\mathcal{L}}_{12}\oplus((T_{0}\mathcal{L}_{12})^\perp\times\{0\})\oplus(\{0\}\times T_{(u_1,u_1')(s)}\Delta_{F_1}).
    \]
    Under this decomposition, an element $x\in X$ can be written as 
    \begin{equation}\label{equ:vecdecomp}
        x=x_C+x_{12}+(x_{12}^\perp,0)+(0,x_{11}).
    \end{equation}
    Define $\Psi:X\supset B_\epsilon\rightarrow X$ as
    \begin{align*}
        \Psi(x):=&x+(\psi^\perp_{12}(x_{12}),0)+(0,\psi_{11}(x_{12}))\\
        =&x_c+x_{12}+(x_{12}^\perp+\psi^\perp_{12}(x_{12}),0)+(0,x_{11}+\psi_{11}(x_{12})).
    \end{align*}
    Then $\Psi(0)=0$ and $d\Psi(0)=id$. Since 
    \[
    \Psi(x_{12})=x_{12}+(\psi^\perp_{12}(x_{12}),0)+(0,\psi_{11}(x_{12})),
    \]
    it can be concluded that $\Psi(T_{0}\hat{\mathcal{L}}_{12})=\hat{\mathcal{L}}_{12}$. These two facts show that 
    \[
    T_{0}(\Psi^{-1}(\hat{\mathcal{L}}_{12}))=d\Psi(0)^{-1}T_{0}\hat{\mathcal{L}}_{12}=T_{0}\hat{\mathcal{L}}_{12}.
    \]
    So $\Psi^{-1}(\hat{\mathcal{L}}_{12})$ can be represented as the graph of a map
    \[
    \varphi=(\varphi^\perp_{12},\varphi_{11}):T_{0}\hat{\mathcal{L}}_{12}\supset B_\epsilon\rightarrow (T_{0}\mathcal{L})^\perp\times(T_{(u_1,u_1')(s)}\Delta_{F_1})
    \]
    with $\varphi(0)=0$, $d\varphi(0)=0$, and $\varphi|_{T_{u_2'(s)}\hat{\mathcal{L}}_{12}}=\psi$.\\
    Using the decomposition in Equation (\ref{equ:vecdecomp}), $Q_s$ can be defined as
    \begin{equation*}
         Q_s(x)=(\varphi_{12}^\perp(x_C+x_{12}),0)+(0,\varphi_{11}(x_C+x_{12}))).
    \end{equation*}
    Then
    \begin{equation*}
        (1+Q_s)(x_C+x_{12})=x_C+x_{12}+(\varphi_{12}^\perp(x_C+x_{12}),0)+(0,\varphi_{11}(x_C+x_{12})))\in\mathcal{L}.
    \end{equation*}
    Moreover,
    \[
    (1+Q_s)(x_{12}+x_{12}^\perp+x_{11})=x_{12}+x_{12}^\perp+x_{11}+(\psi_{12}^\perp(x_{12}),0)+(0,\psi_{11}(x_{12})))\in \Delta.
    \]
    Therefore, the second and the third conditions are satisfied.\\
    Finally one can define $Q^2_{s,t}$. The map $Q^2_{s,1}$ is defined as first restrict $Q_s$ on $F_2\times\{0\}$ then project to the $F_2$ component. To define $Q^2_{s,0}$, notice that 
    \[
    T_{u_2(s,0)}F_2=T_{u_2(s,0)}(L_1\circ F)\oplus T_{u_2(s,0)}(L_1\circ F)^\perp.
    \]
    So as before, the inverse image of $L_1\circ F$ under the exponential map in $T_{u_2(s,0)}F_2$ can be represented by the graph of a map
    \[
    \phi:T_{u_2(s,0)}(L_1\circ F)\rightarrow T_{u_2(s,0)}(L_1\circ F)^\perp.
    \]
    Let $(x_2,x^\perp_2)\in T_{u_2(s,0)}(L_1\circ F)\oplus T_{u_2(s,0)}(L_1\circ F)^\perp$,
    \[
    Q^2_{s,0}(x_2,x^\perp_2)=(0,\phi(x_2)).
    \]
    Set $Q^2_{s,t}=(1-t)Q^2_{s,0}+tQ^2_{s}$, the last two conditions are satisfied.
\end{proof}

Denote $e_{u_2}$ and $e_{(u_2',u_1,u_1')}$ as $\exp_{u_2}\circ(1+Q^2)$ and $\exp_{(u_2',u_1,u_1')}\circ(1+Q)$, respectively. Let $e_u:=(e_{u_2},e_{(u_2',u_1,u_1')})$ Because $e_u(0)=u$, then $de_u(0)$ maps $u^*T(F_2\times F_2\times F_1\times F_1)$ to $T(F_2\times F_2\times F_1\times F_1)|_{\im u}$. By identifying the range of $de_u(0)$ with $u^*T(F_2\times F_2\times F_1\times F_1)$, we get the map
\[
de_u^*(0):u^*T(F_2\times F_2\times F_1\times F_1)\rightarrow u^*T(F_2\times F_2\times F_1\times F_1),
\]
which is identity.\\

\textbf{The next goal} is to find $\eta=(\eta_2,\eta_2',\eta_1,\eta_1')$ such that $e_u(\eta)$ is $J$-holomorphic. However, the domain and the image of the map $\bar{\partial}_Je_u(\eta)$ are different. The solution is to apply the parallel transport defined described as follows. Let $\Phi_{u_2(s,t)}$ be the parallel transport on $T(F_2)$ along the path $\tau\mapsto e_{u_2(s,t)}(\tau\eta_2)$ using the Levi-Civita connection and $\Phi_{u_2(s,t)}$ be the parallel transport on $T(F_2\times F_1\times F_1)$ along the path $\tau\mapsto e_{(u_2',u_1,u_1')(s,t)}(\tau(\eta_2',\eta_1,\eta_1')))$ using the Hermitian connection compatible with the $J$-holomorphic structure. Then with fixed $(s,t)$ one can identify $T_{u(s,t)}(F_2\times F_2\times F_1\times F_1)$ with $T_{e_{u(s,t)}(\eta)}(F_2\times F_2\times F_1\times F_1)$. Applying this identification followed with the \textbf{identification map $I$ induced by the bundle pull back using $u$}, we have a map defined for all $(s,t)$ with fixing target:
\begin{align*}
    \mathcal{F}_u:u^*T(F_2\times F_2\times F_1\times F_1)&\rightarrow u^*T(F_2\times F_2\times F_1\times F_1)\\
    \mathcal{F}_u(\eta_2,\eta_2',\eta_1,\eta_1')=I(\Phi_{u_2}(\eta_2)^{-1}(\bar{\partial}_{J_2}e_{u_2}(\eta_2)),&\Phi_{(u_2',u_1,u_1')}(\eta_2',\eta_1,\eta_1')^{-1}(\bar{\partial}_{(J_2,J_1,J_1)}e_{u_2}(\eta_2',\eta_1,\eta_1')),
\end{align*}
where $J_1$ and $J_2$ are $J$-holomorphic structures on $F_1$ and $F_2$ respectively. \textbf{Then $e_u(\eta)$ is $J$-holomorphic if and only if $\mathcal{F}_u(\eta)=0$}.\\

To show that there is $\eta$ such that $\mathcal{F}_u(\eta)=0$, \textbf{the strategy is to apply the implicit function theorem for functional spaces. The followings first give the necessary definitions of functional spaces and then prove some inequalities}.

Next {we need to set up functional spaces and extend $\mathcal{F}_u$ onto these spaces}. For a positive integer $k$, let $H^k_{1,\delta}$ be the space of tuples $\eta:=(\eta_2,\eta_2',\eta_1,\eta_1')$, where
\[
\eta_2\in H^k(\R\times[0,1],u^*_2TF_2),\quad \eta_2'\in H^k(\R\times[0,\delta],(u_2')^*TF_2),\quad \eta_1,\eta_1'\in H^k(\R\times[0,1],(u_1')^*TF_1).
\]
The norm of $\eta$ on $H^k_{1,\delta}$ is defined as 
\[
\|(\eta_2,\eta_2',\eta_1,\eta_1')\|_{H^k_{1,\delta}}^2=\|\eta_2\|_{H^k(\R\times[0,1])}^2+\|\eta_2'\|_{H^k(\R\times[0,\delta])}^2+\|\eta_1\|_{H^k(\R\times[0,\delta])}^2+\|\eta_1'\|_{H^k(\R\times[0,\delta])}^2.
\]
Let $\Gamma_{1,\delta}$ be the space of $\eta\in H^2_{1,\delta}$ satisfy the following conditions:
\begin{itemize}
    \item $(\eta_2,\eta_2')$ has its $t=0$ boundary in $(u_2',u_2')^*T\Delta_{F_2}$,
    \item $(\eta_1,\eta_1')$ has its $t=0$ boundary in $(u_1',u_1')^*T\Delta_{F_1}$,
    \item $(\eta_2',\eta_1,\eta_1')$ has its $t=\delta$ boundary in $(u_2',u_1,u_1')^*T(L_1\times F)^T$,
    \item $\eta_2$ has its $t=1$ boundary in $u_2^*TL_2$.
\end{itemize}
The norm equipped on $\Gamma_{1,\delta}$ is defined as
\[
\|\eta\|_{\Gamma_{1,\delta}}=\|\eta\|_{H^2_{1,\delta}}+\|\nabla\eta\|_{L^4_{1,\delta}}:=\|\eta\|_{H^2_{1,\delta}}+(\|\nabla\eta_2\|^4_{L^4({\R\times[0,1]})}+\|\nabla(\eta_2',\eta_1,\eta_1'))\|^4_{L^4{(\R\times[0,\delta]})})^\frac14.
\]
Denote the $\epsilon$-ball in $\Gamma_{1,\delta}$ by
\[
\Gamma_{1,\delta}(\epsilon)=\{\eta\in \Gamma_{1,\delta}\||\eta\|_{\Gamma_{1,\delta}}<\epsilon\}.
\]
Let $\Omega_{1,\delta}=H^1_{1,\delta}$ equipped with the norm
\[
\|\eta\|_{\Omega_{1,\delta}}=\|\eta\|_{H^1_{1,\delta}}+\|\eta\|_{L^4_{1,\delta}}.
\]
Therefore, with small enough $\eta$, the map $\mathcal{F}_u$ can be extended to
\[
\mathcal{F}_u:\Gamma_{1,\delta}(\epsilon)\rightarrow \Omega_{1,\delta}.
\]

Given a $J$-holomorphic strip in $F_2$, the corresponding maps $u_2',u_1,u_1'$ are constant in $[0,\delta]$ direction. Therefore, one can identify $(u_2',u_1,u_1')$ with $(u_2',u_1,u_1')|_{\R\times\{0\}}$. Define
\begin{align*}
    \pi^\perp_{211}:C^\infty(\R,(u_2',u_1,u_1')^*T(F_2\times F_1\times F_1))&\rightarrow C^\infty(\R,(u_2',u_1,u_1')^*T((L_1\times F)^T)^\perp),\\
    \pi_{2}:C^\infty(\R,u_2'^*TF_2)&\rightarrow C^\infty(\R,u_2'^*T(L_1\circ F)),\\
    \pi_{2}^\perp:C^\infty(\R,u_2'^*TF_2)&\rightarrow C^\infty(\R,u_2'^*T(L_1\circ F)^\perp)
\end{align*}
as the projections to the corresponding spaces.

\begin{lem}\label{Lem:3.1.3}
    There exists a constant $C$ such that the following holds.
    \begin{itemize}
        \item [$($a$)$] For every $(x_2,x_1,x_1')\in F_2\times F_1\times F_1$
        \[
        d(x_2,L_1\circ F)\le C(d((x_2,x_1,x_1'),(L_1\circ F)^T)+d(x_1,x_1')).
        \]
        \item [$($b$)$] For every $s\in\R$ and $(\eta_2',\eta_1,\eta_1')\in (u_2',u_2,u_1')^*(s)T(F_2\times F_1\times F_1)$
        \begin{align*}
            |(\eta_2',\eta_1,\eta_1)|\le C(|\pi_2\eta_2'|+|\eta_1-\eta_1'|&+|\pi_{211}^\perp(\eta_2',\eta_1,\eta_1')|,\\
            |\pi_2^\perp\eta_2'|\le C(|\pi_{211}^\perp(\eta_2',\eta_1,\eta_1')|&+|\eta_1-\eta_1'|).
        \end{align*}
        \item [$($c$)$] For every $(\eta_2',\eta_1,\eta_1')\in C^\infty(\R,(u_2',u_2,u_1')^*(s)T(F_2\times F_1\times F_1))$
        \[
        \|(\eta_2',\eta_1,\eta_1')\|_{H^1(\R)}\le C(\|\pi_2\eta_2'\|_{H^1(\R)}+\|\eta_1-\eta_1'\|_{H^1(\R)}+\|\pi_{211}^\perp(\eta_2',\eta_1,\eta_1)\|_{H^1(\R)}),
        \]
        and the same holds with $H^1$ replaced by $C^1$ or $L^p$ for any $p\ge1$. Moreover,
        \begin{align*}
            \|\pi_2^\perp\eta_2'\|_{L^2(\R)}\le C(\|\pi_2^\perp(\eta_2',\eta_1,\eta_1')&\|_{L^2(\R)}+\|\eta_1-\eta_1'\|_{L^2(\R)}),\\
            \|\pi_2^\perp\eta_2'\|_{H^1(\R)}\le C(\|\pi_2^\perp(\eta_2',\eta_1,\eta_1')\|_{H^1(\R)}+&\|\eta_1-\eta_1'\|_{H^1(\R)}+\||\partial_s(u_2',u_2,u_1')|\cdot|(\eta_2',\eta_1,\eta_1')|\|).
        \end{align*}
    \end{itemize}
\end{lem}

\begin{proof}
    Suppose the metric on $F_2\times F_1\times F_1$ is the product metric induced by the given metrics on $F_1$ and $F_2$. For $($a$)$, notice that 
    \begin{align*}
       d(x_2,L_1\circ F)\le&d((x_2,x_1,x_1'),(L_1\times F)^T\cap (F_2\times\Delta_{F_1})) \\
       \le&C(d((x_2,x_1,x_1'),(L_1\times F)^T)+d((x_2,x_1,x_1'),F_2\times\Delta_{F_1})).
    \end{align*}
    For $($b$)$, because $L_1\circ F$ is a Lagrangian immersion,
    \begin{align*}
        &(u_2',u_1,u_1')^*(s)T(F_2\times F_1\times F_1)\\
        =&(u_2',u_1,u_1')^*(s)T(L_1\circ F)\oplus(u_2',u_1,u_1')^*(s)T(L_1\circ F)^\perp\\
        =&(u_2',u_1,u_1')^*(s)T(L_1\circ F)\oplus(u_2',u_1,u_1')^*(s)T(\{0\}\times(T\Delta_{F_1})^\perp)\oplus(u_2',u_1,u_1')^*(s)T(L_1\times F)^\perp).
    \end{align*}
    This isomorphism gives the estimate of the norms on $(\eta_2',\eta_1,\eta_1')\in (u_2',u_1,u_1')^*(s)T(F_2\times F_1\times F_1)$
    \[
    |(\eta_2',\eta_1,\eta_1')|\le C(|\pi_2\eta_2'|+|\eta_1-\eta_1'|+|\pi_{211}^\perp(\eta_2',\eta_1,\eta_1')|.
    \]
    Since the composition of the maps 
    \[
    (\{0\}\times(T\Delta_{F_1})^\perp)\oplus T(L_1\times F)^\perp\rightarrow T(L_1\circ F)\oplus T(L_1\circ F)^\perp=TF_2\rightarrow T(L_1\circ F)^\perp
    \]
    is bounded. So after pulling back the the boundaries of the strip, one has the following estimate
    \[
    |\pi_2^\perp\eta_2'|\le C(|\pi_{211}^\perp(\eta_2',\eta_1,\eta_1')|+|\eta_1-\eta_1'|).
    \]
    For $($c$)$,
    \[
    \|\pi_2^\perp\eta_2'\|_{L^2(\R)}\le C(\|\pi_{211}^\perp(\eta_2',\eta_1,\eta_1')\|_{L^2(\R)}+\|\eta_1-\eta_1'\|_{L^2(\R)})
    \]
    is justified by integrating the second inequality of $($b$)$.
    The same method applies to the first inequality of $($b$)$ shows that
    \begin{equation}\label{ineq:1}
        \|(\eta_2',\eta_1,\eta_1)\|_{L^p(\R)}\le C(\|\pi_2\eta_2'\|_{L^p(\R)}+\|\eta_1-\eta_1'\|_{L^p(\R)}+\|\pi_{211}^\perp(\eta_2',\eta_1,\eta_1')\|_{L^p(\R)}
    \end{equation}
    Since we need to estimate the $H^1$ norm, the inequalities in $($b$)$ can be applied to $\nabla_s(\eta_2',\eta_1,\eta_1)(s)$ to get
    \begin{align}\label{ineq:2}
      |\nabla_s(\eta_2',\eta_1,\eta_1)|\le C(|\pi_2(\nabla_s\eta_2')|+|\nabla_s\eta_1-\nabla_s\eta_1'|&+|\pi_{211}^\perp(\nabla_s(\eta_2',\eta_1,\eta_1'))|,\\
            |\pi_2^\perp(\nabla_s\eta_2')|\le C(|\pi_{211}^\perp(\nabla_s(\eta_2',\eta_1,\eta_1'))|&+|\nabla_s\eta_1-\nabla_s\eta_1'|).  
    \end{align}
    The following inequalities are needed
    \begin{align*}
        &|\pi_2(\nabla_s\eta_2')|\le C(|\nabla_s(\pi_2(\eta_2'))|+|(\eta_2',\eta_1,\eta_1)|),\\
        |\pi_{211}^\perp(\nabla_s(\eta_2',\eta_1,\eta_1'))|\le &C(|\nabla_s(\pi_{211}^\perp(\eta_2',\eta_1,\eta_1'))|+|\partial_s(u_2',u_1,u_1')|\cdot|(\eta_2',\eta_1,\eta_1')|),\\
        |\nabla_s(\pi_2^\perp(\eta_2'))&|\le C(|\pi_2^\perp(\nabla_s\eta_2')|+|\partial_s(u_2',u_1,u_1')|\cdot|(\eta_2',\eta_1,\eta_1')|).
    \end{align*}
    The proof of these three inequalities are similar, we only prove the first one. There is an orthonormal frame given by $\gamma_1(s)\in(u_2')^*T(L_1\circ F)$ and $\gamma_2(s)\in(u_2')^*T(L_1\circ F)^\perp$. Suppose in this coordinate, $\eta=\sum\lambda^i\gamma_i$. Therefore
    \begin{align*}
        |\pi_2(\nabla_s\eta_2')-\nabla_s(\pi_2(\eta_2'))|=&|\lambda^1\pi_{2}(\nabla_s\gamma_1)+\lambda^2(\pi_{2}(\nabla_s\gamma_2)-\nabla_s\gamma_2)|\\
        \le &C|\partial_su_2'|\cdot|\eta_2'|.
    \end{align*}
    Notice that $|\partial_su_2'|$ is bounded, this implies that 
    \[
    |\pi_2(\nabla_s\eta_2')|\le C(|\nabla_s(\pi_2(\eta_2'))|+|(\eta_2',\eta_1,\eta_1)|).
    \]
    After integrating Equation (\ref{ineq:2}) with $p$-norm and applying the inequalities just proved,
    \begin{align*}
        \|\nabla_s(\eta_2',\eta_1,\eta_1)\|_{L^p(\R)}\le &C(\|\nabla_s(\pi_2(\eta_2'))\|_{L^p(\R)}+\|\nabla_s\eta_1-\nabla_s\eta_1'\|_{L^p(\R)}\\
        +&\|\nabla_s(\pi_{211}^\perp(\eta_2',\eta_1,\eta_1'))\|_{L^p(\R)}+\|(\eta_2',\eta_1,\eta_1')\|_{L^p(\R)}).
    \end{align*}
    Combine this with Equation (\ref{ineq:1}) to get
    \begin{align*}
        \|(\eta_2',\eta_1,\eta_1')\|_{H^1(\R)}\le C(\|\pi_2\eta_2'\|_{H^1(\R)}+\|\eta_1-\eta_1'\|_{H^1(\R)}+\|\pi_{211}^\perp(\eta_2',\eta_1,\eta_1)\|_{H^1(\R)}).
    \end{align*}
    The last inequality is proved similarly.
\end{proof}

\begin{lem}\label{lem:sobest}
    There is a constant $C$ such that for all $\delta\in(0,1]$ and $\eta=(\eta_2,\eta_2',\eta_1,\eta_1')\in H^2_{1,\delta}$
    \begin{align*}
        &\|\eta_2\|_{C^0([0,1],H^1(\R))}+\|(\eta_2',\eta_1,\eta_1')\|_{C^0([0,\delta],H^1(\R))}\\
        \le &C(\|(\eta_2,\eta_2',\eta_1,\eta_1')\|_{H^2_{1,\delta}}+\|(\eta_2-\eta_2')|_{t=0}\|_{H^1(\R)}+\|(\eta_1-\eta_1')|_{t=0}\|_{H^1(\R)}+\|\pi_{211}^\perp(\eta_2',\eta_1,\eta_1')|_{t=\delta}\|_{H^1(\R)}).
    \end{align*}
    In particular, for all $p>2$ including $p=\infty$ and for $\eta\in\Gamma_{1,\delta}$ satisfying the boundary condition $(\ref{equ:bddcon})$,
    \[
    \|\eta_2\|_{L^p(\R\times[0,1])}+\|(\eta_2',\eta_1,\eta_1')\|_{L^p(\R\times[0,1])}\le C\|(\eta_2,\eta_2',\eta_1,\eta_1')\|_{H^2_{1,\delta}}.
    \]
\end{lem}

\begin{proof}
    The Sobolev embedding implies that $H^1(\R\times[0,1])\rightarrow L^p(\R\times[0,1]))$ for $p\ge2$ (Theorem 5.4 Case B \cite{adams2003sobolev}). Therefore the second inequality is a direct consequence of the first one. We only need to prove the first one. Suppose that on the contrary, there are sequences $\delta^\nu>0$ and $\eta^\nu=(\eta^\nu_2,\eta_2'^\nu,\eta_1^\nu,\eta_1'^\nu)\in H^2_{1,\delta^\nu}$ with 
    \[
    \|\eta^\nu_2\|_{C^0([0,1],H^1(\R))}+\|(\eta_2'^\nu,\eta_1^\nu,\eta_1'^\nu)\|_{C^0([0,\delta^\nu],H^1(\R))}=1
    \]
    but 
    \[
    \|(\eta_2^\nu,\eta_2'^\nu,\eta_1^\nu,\eta_1'^\nu)\|_{H^2_{1,\delta}}+\|(\eta_2^\nu-\eta_2'^\nu)|_{t=0}\|_{H^1(\R)}+\|(\eta_1^\nu-\eta_1'^\nu)|_{t=0}\|_{H^1(\R)}+\|\pi_{211}^\perp(\eta_2'^\nu,\eta_1^\nu,\eta_1'^\nu)|_{t=\delta}\|_{H^1(\R)}\rightarrow0.
    \]
By the Sobolev embedding
\[
H^2(\R\times[0,1])\subset H^1([0,1],H^1(\R))\hookrightarrow C^0([0,1],H^1(\R)),
\]
this implies $\|\eta^\nu_2\|_{C^0([0,1],H^1(\R))}\rightarrow0$, and so
\[
\|{\eta_2'}^\nu|_{t=0}\|_{H^1(\R)}\le\|{\eta_2}^\nu|_{t=0}\|_{H^1(\R)}+\|(\eta_2^\nu-\eta_2'^\nu)|_{t=0}\|_{H^1(\R)}\rightarrow0.
\]
Since
\begin{align*}
    &\|(\eta_2'^\nu,\eta_1^\nu,\eta_1'^\nu)|_{t=t_0}-(\eta_2'^\nu,\eta_1^\nu,\eta_1'^\nu)|_{t=\delta^\nu}\|^2_{H^1(\R)}\\
    \le &\delta^\nu\int_0^{\delta^\nu}\|\nabla_t(\eta_2'^\nu,\eta_1^\nu,\eta_1'^\nu)\|_{H^1(\R)}\\
    \le &\delta^\nu\|(\eta_2'^\nu,\eta_1^\nu,\eta_1'^\nu)\|^2_{H^2(\R\times[0,\delta
    ^\nu])}\rightarrow0, \quad \forall t_0\in[0,\delta^\nu],
\end{align*}
then the previous lemma shows that
\begin{align*}
    &\|(\eta_2'^\nu,\eta_1^\nu,\eta_1'^\nu)|_{t=\delta^\nu}\|^2_{H^1(\R)}\\
    \le &C(\|\pi_2\eta_2'^\nu|_{t=\delta^\nu}\|_{H^1(\R)}+\|(\eta_1^\nu-\eta_1'^\nu)|_{t=\delta^\nu}\|_{H^1(\R)}+\|\pi_{211}^\perp(\eta_2'^\nu,\eta_1^\nu,\eta_1'^\nu)|_{t=\delta^\nu}\|_{H^1(\R)})\\
    \le &C(\|\pi_2{\eta_2'}^\nu|_{t=0}\|_{H^1(\R)}+\|\pi_2(\eta_2'^\nu|_{t=\delta^\nu}-\eta_2'^\nu|_{t=0})\|_{H^1(\R)}+\|(\eta_1^\nu-{\eta_1'}^\nu)|_{t=0}\|_{H^1(\R)}\\
    +&2\|(\eta_2'^\nu,\eta_1^\nu,\eta_1'^\nu)|_{t=\delta^\nu}-(\eta_2'^\nu,\eta_1^\nu,\eta_1'^\nu)|_{t=0}\|_{H^1(\R)}+\|\pi_{211}^\perp(\eta_2'^\nu,\eta_1^\nu,\eta_1'^\nu)|_{t=\delta^\nu}\|_{H^1(\R)})\\
    \rightarrow &0.
\end{align*}
Combine this with the fact that
\[
\|(\eta_2'^\nu,\eta_1^\nu,\eta_1'^\nu)|_{t=t_0}-(\eta_2'^\nu,\eta_1^\nu,\eta_1'^\nu)|_{t=\delta^\nu}\|^2_{H^1(\R)}\rightarrow0,
\]
this gives that
\[
\|(\eta_2'^\nu,\eta_1^\nu,\eta_1'^\nu)\|_{C^0([0,\delta^\nu],H^1(\R))}\rightarrow0.
\]
This and
\[
\|\eta^\nu_2\|_{C^0([0,1],H^1(\R))}\rightarrow0
\]
contradicts with the assumption that
\[
\|\eta^\nu_2\|_{C^0([0,1],H^1(\R))}+\|(\eta_2',\eta_1,\eta_1')\|_{C^0([0,\delta^\nu],H^1(\R))}=1.
\]
\end{proof}
Given $\eta\in\Gamma_{1,\delta}(\epsilon)$, \textbf{the next task is to compare $d\mathcal{F}_u(\eta)$ with $D_{e_u(\eta)}$}, where $D_{e_u(\eta)}$ is the linearized operator of the Cauchy-Riemann operator on $F_2\times F_2\times F_1\times F_1$.

\begin{lem}\label{lem:linearest}
    There are uniform constants $\epsilon>0$ and $C$ such that for all $\delta\in(0,1]$ and $\eta\in\Gamma_{1,\delta}(\epsilon),\alpha\in\Gamma_{1,\delta}$
    \begin{align*}
        \|\mathcal{F}_u(0)\|_{\Omega_{1,\delta}}&\le C\delta^\frac14,\\
        \|d\mathcal{F}_u(\eta)\alpha-d\mathcal{F}_u(0)\alpha\|_{\Omega_{1,\delta}}&\le C\|\eta\|_{\Gamma_{1,\delta}}\|\alpha\|_{\Gamma_{1,\delta}},\\
        \|d\mathcal{F}_u(\eta)\alpha-I\Phi_u(\eta)^{-1}D_{e_u(\eta)}E_u(\eta)\alpha\|_{\Omega_{1,\delta}}&\le C\|\eta\|_{\Gamma_{1,\delta}}\|\alpha\|_{\Gamma_{1,\delta}},
    \end{align*}
    where
    \[
    E_u(\eta)\alpha:=\frac{d}{dt}e_u(\eta+t\alpha)|_{t=0}.
    \]
\end{lem}

\begin{proof}
    Since the tuple $(u_2',u_1,u_1')$ is the constant extension of $u_2|_{t=0}$ along $t$ direction and $u_2$ is pseduoholomorphic, therefore
    \begin{align*}
         \|\mathcal{F}_u(0)\|_{\Omega_{1,\delta}}=&\|I\partial_s(0,u_2',u_1,u_1')\|_{H^1_{1,\delta}}+\|I\partial_s(0,u_2',u_1,u_1')\|_{L^4_{1,\delta}}\\
        &\le C\delta^\frac12(\|\partial_su_2'|_{t=0}\|_{H^1(\R)}+2\|\partial_su_1|_{t=0}\|_{H^1(\R)})\\
        &+C\delta^\frac14(\|\partial_su_2'|_{t=0}\|_{L^4(\R)}+2\|\partial_su_1|_{t=0}\|_{L^4(\R)})\\
        &=C\delta^\frac12(\|\partial_su_2|_{t=0}\|_{H^1(\R)}+2\|\partial_su_1|_{t=0}\|_{H^1(\R)})\\
        &+C\delta^\frac14(\|\partial_su_2|_{t=0}\|_{L^4(\R)}+2\|\partial_su_1|_{t=0}\|_{L^4(\R)})\\
        &\le C\delta^{\frac14}.
    \end{align*}
    After taking the Hermitian covariant derivative $\Tilde{\nabla}$ at $t=0$ for the identity
    \[
    \Phi_u(\eta+t\alpha) I^{-1}\mathcal{F}_u(\eta+t\alpha)=\bar{\partial}_J(e_u(\eta+t\alpha)),
    \]
    we have
    \[
    \Phi_u(\eta)I^{-1}d\mathcal{F}_u(\eta)\alpha-D_{e_u(\eta)}E_u(\eta)\alpha=-\Psi_u(\eta,\alpha,\mathcal{F}_u(\eta)),
    \]
    where
    \[
    \Psi_u(\eta,\alpha,\mathcal{F}_u(\eta)):=\Tilde{\nabla}_t(\Phi_u(\eta+t\alpha)I^{-1}\mathcal{F}_u(\eta))|_{t=0}.
    \]
    So for the third inequality, we only need to estimate $\|\Psi_u(\eta,\alpha,\mathcal{F}_u(\eta))\|_{\Omega_{1,\delta}}$, This estimate is contained in the below.

    Let $\eta\in\Gamma_{1,\delta}(\epsilon)$ and $\alpha\in\Gamma_{1,\delta}$, then Lemma (\ref{lem:sobest}) implies that 
    \[
    \|(\eta_2',\eta_1,\eta_1')\|_{C^0}\le C\|(\eta_2',\eta_1,\eta_1')\|_{H^2_{1,\delta}}\le C\epsilon, \|(\alpha_2',\alpha_1,\alpha_1')\|_{C^0}\le C\|(\alpha_2',\alpha_1,\alpha_1')\|_{H^2_{1,\delta}}.
    \]
    Consider 
    \begin{align*}
        &E_{(u_2',u_1,u_1')}(\eta_2',\eta_1,\eta_1')(\alpha_2',\alpha_1,\alpha_1'):=\frac{d}{dt}e_{(u_2',u_1,u_1')}((\eta_2',\eta_1,\eta_1')+t(\alpha_2',\alpha_1,\alpha_1'))|_{t=0}\\
        &\Psi_{(u_2',u_1,u_1')}((\eta_2',\eta_1,\eta_1'),(\alpha_2',\alpha_1,\alpha_1'),(\xi_2',\xi_1,\xi_1'))\\
        :=&\Tilde{\nabla}_t(\Phi_{(u_2',u_1,u_1')}((\eta_2',\eta_1,\eta_1')+t(\alpha_2',\alpha_1,\alpha_1'))(\xi_2',\xi_1,\xi_1'))|_{t=0}.
    \end{align*}
    Notice that $E_{(u_2',u_1,u_1')}(0)=$Id and that $\Psi(0,(\alpha_2',\alpha_1,\alpha_1'),(\xi_2',\xi_1,\xi_1'))=0$ and the fact that these maps are linear in $(\alpha_2',\alpha_1,\alpha_1'),(\xi_2',\xi_1,\xi_1')$ and depend smoothly on $(\eta_2',\eta_1,\eta_1')$. Therefore the following estimates are true
    \begin{align*}
        |E_{(u_2',u_1,u_1')}(\eta_2',\eta_1,\eta_1')|&\le C,\\
        |\nabla E_{(u_2',u_1,u_1')}(\eta_2',\eta_1,\eta_1')|&\le C(|\nabla(\eta_2',\eta_1,\eta_1')|+|d(u_2',u_1,u_1')\|(\eta_2',\eta_1,\eta_1')|),\\
        \Psi_{(u_2',u_1,u_1')}((\eta_2',\eta_1,\eta_1'),(\alpha_2',\alpha_1,\alpha_1'),(\xi_2',\xi_1,\xi_1'))&\le C|(\eta_2',\eta_1,\eta_1')\|(\alpha_2',\alpha_1,\alpha_1')\|(\xi_2',\xi_1,\xi_1')|.
    \end{align*}
    Set $\bar{u}=(u_2',u_1,u_1')$,$\hat{\eta}=(\eta_2',\eta_1,\eta_1')$,$\hat{\alpha}=(\alpha_2',\alpha_1,\alpha_1')$, the following equation is true (Keeping track on the estimates of the first term on the right hand side below gives the third inequality of this lemma.)
    \begin{align*}
        &\Phi_{\bar{u}}(\hat{\eta})I^{-1}(d\mathcal{F}_{\bar{u}}(\hat{\eta})\hat{\alpha}-d\mathcal{F}_{\bar{u}}(0)\hat{\alpha})\\
        =&-\Psi_{\bar{u}}(\hat{\eta},\hat{\alpha},\mathcal{F}_{\bar{u}}(\hat{\eta}))+(\nabla(E_{\bar{u}}(\hat\eta))\hat{\alpha})^{0,1}+((E_{\bar{u}}(\hat\eta)-\Phi_{\bar{u}}(\hat{\eta}))\nabla\hat\alpha)^{0,1}\\
        &-\frac12J(e_{\bar{u}}(\hat\eta))(((\nabla_{(E_{\bar{u}}(\hat\eta)-\Phi_{\bar{u}}(\hat{\eta}))\hat\alpha}J)(e_{\bar{u}}(\hat\eta)))\Phi_{\bar{u}}(\hat\eta)d\bar{u})^{0,1}\\
        &-\frac12J(e_{\bar{u}}(\hat\eta))(((\nabla_{\Phi_{\bar{u}}(\hat{\eta})\hat\alpha}J)(e_{\bar{u}}(\hat\eta))-\Phi_{\bar{u}}(\hat\eta)(\nabla_{\hat\alpha}J)(\bar{u})\Phi_{\bar{u}}(\hat\eta)^{-1})\Phi_{\bar{u}}(\hat\eta)d\bar{u})^{0,1}\\
        &-\frac12J(e_{\bar{u}}(\hat\eta))((\nabla_{E_{\bar{u}}(\hat{\eta})\hat\alpha}J)(e_{\bar{u}}(\hat\eta))(d(e_{\bar{u}}(\hat\eta))-\Phi_{\bar{u}}(\hat\eta)d\bar{u}))^{0,1}.
    \end{align*}
    After applying $I\Phi_{\bar{u}}(\hat\eta)^{-1}$ to the above identity, using the fact that $|J|,|\nabla J|,|I\Phi_{\bar{u}}(\hat\eta)^{-1}|,|d\bar{u}|,|\hat\eta|$ are bounded, and
    \begin{align*}
        |\mathcal{F}_{\bar{u}}|\le C|d(e_{\bar{u}}(\hat\eta))|\le C(|\nabla\hat\eta|+|d\bar{u}|),|d(e_{\bar{u}}(\hat\eta))-\Phi_{\bar{u}}(\hat\eta)d\bar{u}|\le C(|\nabla\hat\eta|+|d\bar{u}\|\hat\eta|),\\
        |E_{\bar{u}}(\hat\eta)-\Phi_{\bar{u}}(\hat\eta)|\le C|\hat\eta|,|(\nabla_{\Phi_{\bar{u}}(\hat{\eta})\hat\alpha}J)(e_{\bar{u}}(\hat\eta))-\Phi_{\bar{u}}(\hat\eta)(\nabla_{\hat\alpha}J)(\bar{u})\Phi_{\bar{u}}(\hat\eta)^{-1}|\le C|\hat\eta\|\hat\alpha|,
    \end{align*}
    then the above equation gives the following estimate:
    \begin{align*}
        &|d\mathcal{F}_{\bar{u}}(\hat{\eta})\hat{\alpha}-d\mathcal{F}_{\bar{u}}(0)\hat{\alpha}|\\
        \le &C(|\hat\eta\|\hat\alpha|+|\hat\alpha\|\nabla\hat\eta|)+C(|\nabla\hat\eta\|\hat\alpha|+|\hat\eta\|\hat\alpha|)\\
        +&C(|\hat\eta\|\nabla\alpha|)+C(|\hat\eta\|\nabla\hat\alpha|)+C(|\hat\eta\|\hat\alpha|)+C(|\hat\alpha\|\nabla\hat\eta|+|\hat\alpha\|\hat\eta|)\\
        \le &C(|\hat\eta\|\hat\alpha|+|\hat\alpha\|\nabla\hat\eta|+|\hat\eta\|\nabla\hat\alpha|).
    \end{align*}
    Using Holder's inequality and Lemma (\ref{lem:sobest}), we get the following estimates
    \begin{align*}
        \|d\mathcal{F}_{\bar{u}}(\hat{\eta})\hat{\alpha}-d\mathcal{F}_{\bar{u}}(0)\hat{\alpha}\|_{L^2}\le &C(\|\hat\eta\|_{L^4}\|\hat\alpha\|_{L^4}+\|\hat\alpha\|_{C^0}\|\nabla\hat\eta\|_{L^2}+\|\hat\eta\|_{C^0}\|\nabla\hat\alpha\|_{L^2})\\
        \le &C\|\hat\eta\|_{H^2_{1,\delta}}\|\hat\alpha\|_{H^2_{1,\delta}},\\
        \le &C\|\hat\eta\|_{\Gamma_{1,\delta}}\|\hat\alpha\|_{\Gamma_{1,\delta}},\\
        \|d\mathcal{F}_{\bar{u}}(\hat{\eta})\hat{\alpha}-d\mathcal{F}_{\bar{u}}(0)\hat{\alpha}\|_{L^4}\le &C(\|\hat\eta\|_{L^8}\|\hat\alpha\|_{L^8}+\|\hat\alpha\|_{C^0}\|\nabla\hat\eta\|_{L^4}+\|\hat\eta\|_{C^0}\|\nabla\hat\alpha\|_{L^4})\\
        \le &C(\|\hat\eta\|_{H^2_{1,\delta}}+\|\hat\eta\|_{L^4_{1,\delta}})(\|\hat\alpha\|_{H^2_{1,\delta}}+\|\nabla\hat\alpha\|_{L^4_{1,\delta}})\\
        = &C\|\hat\eta\|_{\Gamma_{1,\delta}}\|\hat\alpha\|_{\Gamma_{1,\delta}}.
    \end{align*}
    It remains to estimate the norm of the derivative of $d\mathcal{F}_{\bar{u}}(\hat{\eta})\hat{\alpha}-d\mathcal{F}_{\bar{u}}(0)\hat{\alpha}$.
    \begin{align*}
    &\|\nabla(d\mathcal{F}_{\bar{u}}(\hat{\eta})\hat{\alpha}-d\mathcal{F}_{\bar{u}}(0)\hat{\alpha})\|_{L^2}\\
    \le &C\|(|\hat\eta|+|\nabla\hat\eta|)(|\hat\alpha|+|\nabla\hat\alpha|)\|_{L^2}\\
    &+C\|(|\nabla^2\hat\eta\|\hat\alpha|+|\nabla\hat\eta|^2|\hat\alpha|+|\nabla\hat\eta\|\nabla\hat\alpha|+|\hat\eta\|\nabla^2\hat\alpha|)\|_{L^2}\\
    \le &C(\|\hat\eta\|_{L^4}+\|\nabla\hat\eta\|_{L^4})(\|\hat\alpha\|_{L^4}+\|\nabla\hat\alpha\|_{L^4})\\
    &+C(\|\nabla^2\hat\eta\|_{L^2}\|\hat\alpha\|_{C^0}+\|\nabla\hat\eta\|_{L^4}^2\|\hat\alpha\|_{C^0}+\|\nabla\hat\eta\|_{L^4}\|\nabla\hat\alpha\|_{L^4}+\|\hat\eta\|_{C^0}\|\nabla^2\hat\alpha\|_{L^2})\\
    &(\text{Lemma } (\ref{lem:sobest})\,\|\nabla\hat\eta\|_{\Gamma_{1,\delta}}\le\epsilon)\\
    \le &C\|\hat\eta\|_{\Gamma_{1,\delta}}\|\hat\alpha\|_{\Gamma_{1,\delta}}
\end{align*}
\end{proof}

To prove the main theorem, the following theorem is needed.
\begin{them}[McDuff and Salamon \cite{mcduff2012j} Proposition A 3.4]\label{thm:mcdu}
    Let $X$ and $Y$ be Banach spaces, $U\subset X$ be an open set, and $f:U \rightarrow Y$ be a continuously differentiable map. Let $x_0\in U$ be such that $D:=df(x_0):X\rightarrow Y$ is surjective and has a $($ bounded linear $)$ right inverse $Q:Y\rightarrow X$. Choose positive constants $\delta$ and $c$ such that $\|Q\|\le c$, $B_\delta(x_0;X)\subset U$, and 
    \[
    \|x-x_0\|<\varepsilon\Rightarrow \|df(x)-D\|\le\frac1{2c},
    \]
    Suppose that $x_1\in X$ satisfies
    \[
    \|f(x_1)\|<\frac{\varepsilon}{4c},\quad \|x_1-x_0\|<\frac\varepsilon8,
    \]
    Then there exists a unique $x\in X$ such that 
    \[
    f(x)=0,\quad x-x_1\in\im Q,\quad \|x-x_0\|<\varepsilon.
    \]
    Moreover, $\|x-x_1\|\le2c\|f(x_1)\|$.
\end{them}

In our case, $f=\mathcal{F}_u$ and $x_0=0$. According to Lemma $(\ref{lem:linearest})$, $\|\eta\|$ and $\|\alpha\|$ can be chosen small enough such that 
\[
\|d\mathcal{F}_u(\eta)\alpha-d\mathcal{F}_u(0)\alpha\|_{\Omega_{1,\delta}}\le\frac1{2c}.
\]
If further more, set $x_1=x_0$ and let $\delta$ small enough, then the third condition in the above theorem is satisfied. So what remains is to show that $d\mathcal{F}_u(0)$ is surjective and has a bounded linear right inverse. With this condition satisfied, there is an $\eta$ such that $\mathcal{F}_u(\eta)=0$ by the above theorem. According to the definition of $\mathcal{F}_u(\eta)$, $e_u(\eta)$ is $J$-holomorphic.

To simplify the notation, let
\begin{align*}
    D^\delta=d\mathcal{F}_u(0):&\Gamma_{1,\delta}\rightarrow\Omega_{1,\delta}\\
    D^\delta\eta=d\mathcal{F}_u(0)\eta=&(D_2\eta_2,D_{211}(\eta_2',\eta_1,\eta_1'))
\end{align*}
where $D_2$ and $D_{211}$ are the linearization of the Cauchy-Riemann operator at $u_2$ and $(u_2',u_1,u_1')$, respectively. Because the map $I$ is identity on the fiber, the same formulas as in Wehrheim and Woodward hold \cite{wehrheim2009floer} here. More precisely,
\begin{align*}
    D_2\eta_2=&\nabla_s\eta_2+J(u_2)\nabla_t\eta_2+\nabla_{\eta_2}J_2(u_2)\partial_tu_2\\
    D_{211}(\eta_2',\eta_1,\eta_1')=&\nabla_s(\eta_2',\eta_1,\eta_1')+J(u_2',u_1,u_1')\nabla_t(\eta_2',\eta_1,\eta_1')\\
    +&\frac{1}{2}\nabla_{(\eta_2',\eta_1,\eta_1')}J(u_2',u_1,u_1')J(u_2',u_1,u_1')\partial_s(\eta_2',\eta_1,\eta_1').
\end{align*}
The formal adjoint operator $(D^\delta)^*$ of $D^\delta$ is given by $(-\nabla_s+J(u_2)\nabla_t,-\nabla_s+J(u_2',u_1,u_1')\nabla_t)$ adding lower order terms. So $(D^\delta)^*$ has the same analytic properties as $D^\delta$.  If $(D^\delta)^*$ is injective, then $D^\delta$ is surjective (\cite{wehrheim2009floer} Page 201).\\



\textbf{Another lemma is needed to prove that $\ker((D^\delta)^*)=0$}.

\begin{lem}
    \begin{itemize}
        \item [$($a$)$] There is a constant $c_1>0$ such that for all $\delta\in(0,1]$ and $\eta\in\Gamma_{1,\delta}$
        \begin{equation}\label{ineq:4}
            \begin{aligned}
            c_1\|\eta\|_{H^2_{1,\delta}}&\le\|D^\delta\eta\|_{H^1_{1,\delta}}+\|\eta\|_{H^0_{1,\delta}}+\|(\eta_2',\eta_1,\eta_1')|_{t=\delta}\|_{H^1(\R)}+\|\eta_2|_{t=1}\|_{H^1(\R)},\\
            c_1\|\nabla\eta\|_{L^4_{1,\delta}}&\le\|D^\delta\eta\|_{H^1_{1,\delta}}+\|D^\delta\eta\|_{L^4_{1,\delta}}+\|\eta\|_{H^0_{1,\delta}}+\|(\eta_2',\eta_1,\eta_1')|_{t=\delta}\|_{H^1(\R)}\\
            &+\|\eta_2|_{t=1}\|_{H^1(\R)},
        \end{aligned}
        \end{equation}
        and the same holds with $D^\delta$ replaced by $(D^\delta)^*$.
        \item [$($b$)$] There is a constant $c_2>0$ such that for all $\delta\in(0,1]$ and $\eta\in\Gamma_{1,\delta}$
        \begin{align*}
            &c_2(\|\hat\eta|_{t=\delta}\|_{H^1(\R)}+\|\eta_2|_{t=1}\|_{H^1(\R)}+\|\eta\|_{H^0_{1,\delta}})\\
            \le &\|D^*_{u_2}\eta_2\|_{H^1(\R\times[0,1])}+\sqrt{\delta}\|\nabla_t\hat\eta\|_{H^1(\R\times[0,\delta])},
        \end{align*}
        and for all $\eta\in\Gamma_{1,\delta}\cap K_0$
        \begin{equation}\label{ineq:3}
            \begin{aligned}
                &c_2(\|(\eta_2',\eta_1,\eta_1')\|_{H^1(\R)}+\|\eta_2|_{t=1}\|_{H^1(\R)}+\|\eta\|_{H^0_{1,\delta}})\\
            \le &\|D_{u_2}\eta_2\|_{H^1(\R\times[0,1])}+\sqrt{\delta}\|\nabla_t(\eta_2',\eta_1,\eta_1')\|_{H^1(\R\times[0,\delta])},
            \end{aligned}
        \end{equation}
        where 
        \[
        K_0:=\{\eta=(\eta_2,\eta_2',\eta_1,\eta_1')\in\Gamma_{1,\delta}|\langle\eta_2,\ker(D_{u_2}\oplus\pi_2^\perp)\rangle_{L^2}=0\}.
        \]
    \end{itemize}
\end{lem}

\begin{proof}
    To prove $($a$)$, define
    \[
        J(\xi)=(J_2(\xi_2),\hat J(\hat\xi))=(de_{u_2}(\xi_2)^{-1}J_2(e_{u_2}(\xi_2))de_{u_2}(\xi_2),de_{\hat u}(\hat\xi)^{-1}J(e_{\hat u}(\hat\xi))de_{u}(\hat\xi)),
    \]
    where $J=(J_2,\hat J)=(J_2,J_2',J_1,J_1')$ is the almost complex structure, $u=(u_2,\hat u)=(u_2,u_2',u_1,u_1)$ is the given strip, and $\xi=(\xi_2,\hat\xi)=(\xi_2,\xi_2',\xi_1,\xi_1')\in\Gamma_{1,\delta}$. If $\xi=(\xi_2,\hat\xi)\in\Gamma_{1,\delta}(\varepsilon)$, then
    \[
    \bar\partial_J(e_u(\xi))=de_u(\xi)(\nabla_s\xi+J(\xi)\nabla_t\xi)+\partial_ue(\xi)\partial_su+J(u)\partial_ue(\xi)\partial_tu.
    \]
    The linearized operator is
    \[
    D^\delta\xi=\nabla_s\xi+J\nabla_t\xi+(\nabla_{\xi_2}J_2(u_2)\partial_tu_2,\frac12\nabla_{\hat\xi}\hat J(\hat u)\hat J(u)\partial_t(u)).
    \]
    Notice that 
    \begin{align*}
        \|\nabla_{\xi_2}J_2(u_2)\partial_tu_2\|_{L^2(\R\times[0,1])}+\|\nabla_{\hat\xi}\hat J(u)J(u)\partial_su\|_{L^2(\R\times[0,\delta])}\le C\|\xi\|_{H^0_{1,\delta}},\\
        \|\nabla_{\hat\xi}\hat J(u)J(u)\partial_su\|_{H^1(\R\times[0,\delta])}\le C\|\xi\|_{H^1_{1,\delta}}.
    \end{align*}
    For $J=J(0)$ and $\eta\in\Gamma_{1,\delta}$, we first show that 
    \[
    \|\eta\|_{H^2_{1,\delta}}\le C(\|\nabla_s\eta+J\nabla_t\eta\|_{H^1_{1,\delta}}+\|\eta\|_{H^0_{1,\delta}}+\|\hat\eta|_{t=\delta}\|_{H^1(\R)}+\|\eta_2|_{t=1}\|_{H^1(\R)}).
    \]
    First it is trivial that 
    \[
    \|\eta\|_{H^0_{1,\delta}}\le C(\|\nabla_s\eta+J\nabla_t\eta\|_{H^1_{1,\delta}}+\|\eta\|_{H^0_{1,\delta}}+\|\hat\eta|_{t=\delta}\|_{H^1(\R)}+\|\eta_2|_{t=1}\|_{H^1(\R)}).
    \]
    Then we need to estimate $\|\nabla\eta\|_{L^2_{1,\delta}}$.\\
    Notice that 
    \begin{align*}
        &\|\nabla_s\eta+J\nabla_t\eta\|_{L^2_{1,\delta}}^2\\
        =&\int(|\nabla_s\eta|^2+|\nabla_t\eta|^2)+\langle\nabla_s\eta,J\nabla_t\eta\rangle-\langle\nabla_t\eta,J\nabla_s\eta\rangle\quad\text{(integral by parts)}\\
        =&\|\nabla\eta\|^2_{L^2_{1,\delta}}-\int(\nabla_sg(\eta,J\nabla_t\eta)-\nabla_tg(\eta,J\nabla_s\eta))-\int(\eta,\nabla_s(J\nabla_t\eta)-\nabla_t(J\nabla_s\eta))\\
        &-\lim_{S\rightarrow\infty}\int_{S=-S}(\eta,J\nabla_t\eta)+\lim_{s\rightarrow\infty}\int_{s=S}(\eta,J\nabla_t\eta)\\
        &+\int_{\R\times\{0\}}(\eta,J\nabla_s\eta)-\int_{\R\times\{1\}}(\eta_2,J_2\nabla_s\eta_2)-\int_{\R\times\{\delta\}}(\hat\eta,\hat J\nabla_s\hat\eta)
    \end{align*}
    Denote the last 3 terms above as $-\mathcal D{(\eta_2|_{t=0},\hat\eta|_{t=0})}$, $\Omega_2(\eta_2|_{t=1})$, and $\Omega_{211}(\hat\eta|_{t=\delta})$, respectively.\\
    According to Holder's inequality,
    \[
    |\int(\nabla_sg(\eta,J\nabla_t\eta)-\nabla_tg(\eta,J\nabla_s\eta))|\le C\|\eta\|_{L^2_{1,\delta}}^2\|\nabla\eta\|_{L^2_{1,\delta}}^2.
    \]
    By using Leibniz's rule and Holder's inequality,
    \begin{align*}
        &|\int(\eta,\nabla_s(J\nabla_t\eta)-\nabla_t(J\nabla_s\eta))|\\
        =&|\int(\eta,\nabla_s(J)\nabla_t\eta+J\nabla_s(\nabla_t\eta)-\nabla_t(J)\nabla_s\eta-J\nabla_t(\nabla_s\eta))|\\
        \le&C\int|\eta||\nabla\eta|. 
    \end{align*}
    Because
    \[
    \int_{s=-S}(\eta,J\nabla_t\eta)\rightarrow0,\quad \int_{s=S}(\eta,J\nabla_t\eta)\rightarrow0,\quad as\quad S\rightarrow0,
    \]
    then
    \[
    \lim_{S\rightarrow\infty}\int_{s=-S}(\eta,J\nabla_t\eta)=\lim_{S\rightarrow\infty}\int_{s=S}(\eta,J\nabla_t\eta)=0.
    \]
    The above combines to get
    \begin{align*}
        &\|\nabla_s\eta+J\nabla_t\eta\|_{L^2_{1,\delta}}^2\\
        \ge&\|\nabla\eta\|_{L^2_{1,\delta}}^2-C\int(|\eta||\nabla\eta|+|\eta|^2)-\mathcal D{(\eta_2|_{t=0},\hat\eta|_{t=0})}-\Omega_2(\eta_2|_{t=1})-\Omega_{211}(\hat\eta|_{t=\delta}).
    \end{align*}
    Since there is a constant such that $|\eta||\nabla\eta|\le C|\eta|^2+\frac{1}{2}|\nabla\eta|^2$, then
    \begin{align*}
      &\|\nabla_s\eta+J\nabla_t\eta\|_{L^2_{1,\delta}}^2\\
      \ge&\|\nabla\eta\|_{L^2_{1,\delta}}^2-C\|\eta\|_{L^2_{1,\delta}}^2-\frac12\|\nabla\eta\|_{L^2_{1,\delta}}^2-\mathcal D{(\eta_2|_{t=0},\hat\eta|_{t=0})}-\Omega_2(\eta_2|_{t=1})-\Omega_{211}(\hat\eta|_{t=\delta})\\
      =&\frac12\|\nabla\eta\|_{L^2_{1,\delta}}^2-C\|\eta\|_{L^2_{1,\delta}}^2-\mathcal D{(\eta_2|_{t=0},\hat\eta|_{t=0})}-\Omega_2(\eta_2|_{t=1})-\Omega_{211}(\hat\eta|_{t=\delta}).
    \end{align*}
    According to the boundary condition of $\eta$,
    \begin{align*}
        &|\mathcal D(\eta_2|_{t=0},\hat\eta|_{t=0})|\\
        =&|\int_{\R\times\{0\}}(\eta,J\nabla_s\eta)|\\
        =&|\int_\R(\omega_2(\eta_2(s,0),\nabla_s\eta_2(s,0))+\omega_{211}(\hat\eta(s,0),\nabla_s\hat\eta(s,0)))|\\
        =&|\int_\R(\omega_2(\eta_2(s,0),\nabla_s\eta_2(s,0))-\omega_2(\eta_2'(s,0),\nabla_s\eta_2'(s,0))\\
        -&\omega_1(\eta_1(s,0),\nabla_s\eta_1(s,0))+\omega_1(\eta_1'(s,0),\nabla_s\eta_1'(s,0)))|\\
        =&0.
    \end{align*}
    Then it is necessary to estimate $\Omega_2(\eta_2|_{t=1}),\Omega_{211}(\hat\eta|_{t=\delta})$. By taking an orthonormal basis $\{\gamma_i(s)\}_i$ on the pull back of the tangent space of the Lagrangians, then
    \begin{align*}
        \eta=\sum\lambda^i\gamma_i,\qquad\nabla_s\eta=\sum(\partial_s\lambda^i\gamma_i+\lambda^i\nabla_s\gamma_i),\\
        \nabla^2_s\eta=\sum(\partial_s^2\lambda^i\gamma_i+2\partial_s\lambda^i\nabla_s\gamma_i+\lambda^i\nabla_s^2\gamma_i).\quad   
    \end{align*}
    Therefore the norm of the tuple $\lambda:=(\lambda_i)_i$ is the same as the norm of $\eta$ and $\omega(\gamma_i,\gamma_j)=0$. These indicate the following inequalities:
    \begin{align*}
        |\int_\R\omega(\eta,\nabla_s\eta)|=&|\int_\R\omega(\sum\lambda^i\gamma_i,\sum(\partial_s\lambda^i\gamma_i+\lambda^i\nabla_s\gamma_i))|\\
        =&|\int_\R\omega(\sum\lambda^i\gamma_i,\sum\lambda^i\nabla_s\gamma_i)|\\
        \le&\int_\R C|\eta(s)||\lambda(s)|ds\\
        =&C\|\eta\|^2_{L^2(\R)},
    \end{align*}
    and
    \begin{align*}
       &|\int_\R\omega(\nabla_s\eta,\nabla_s^2\eta)|\\
       =&|\int_\R\omega(\sum(\partial_s\lambda^i\gamma_i+\lambda^i\nabla_s\gamma_i),\sum(\partial_s^2\lambda^i\gamma_i+2\partial_s\lambda^i\nabla_s\gamma_i+\lambda^i\nabla_s^2\gamma_i))|\\
       =&|\int_\R\omega(\sum\partial_s\lambda^i\gamma_i,\sum(2\partial_s\lambda^i\nabla_s\gamma_i+\lambda^i\nabla_s^2\gamma_i))|\\
       &+|\int_\R\omega(\sum\lambda^i\nabla_s\gamma_i,\sum(\partial_s^2\lambda^i\gamma_i+2\partial_s\lambda^i\nabla_s\gamma_i+\lambda^i\nabla_s^2\gamma_i))|\\
       \le&|C\int_\R(|\nabla_s\eta||\lambda|+|\nabla_s\eta||\partial_s\lambda|+|\partial_s\lambda|^2+|\lambda|^2)|\\
       \le&C\|\eta\|_{H^1(\R)}^2.
    \end{align*}
    This type of inequality shows that
    \begin{align*}
       |\Omega_2(\eta_2|_{t=1})|\le&C\|\eta_2|_{t=1}\|^2_{L^2(\R)},\\
       |\Omega_{211}(\hat\eta|_{t=\delta})|\le&C\|\hat\eta_{211}|_{t=\delta}\|^2_{L^2(\R)},\\
       |\Omega_2(\nabla_s\eta_2|_{t=1})|\le&C\|\eta_2|_{t=1}\|^2_{H^1(\R)},\\
       |\Omega_{211}(\nabla_s\hat\eta|_{t=\delta})|\le&C\|\hat\eta_{211}|_{t=\delta}\|^2_{H^1(\R)}.
    \end{align*}
     Thus
       \begin{equation}\label{ineq:6}
           \begin{aligned}
               \|\nabla\eta\|_{L^2_{1,\delta}}^2\le&C(\|\nabla_s\eta+J\nabla_t\eta\|_{L^2_{1,\delta}}^2+\|\eta\|_{L^2_{1,\delta}}^2+|\Omega_2(\eta_2|_{t=1})|+|\Omega_{211}(\hat\eta|_{t=\delta})|)\\           \le&C(\|\nabla_s\eta+J\nabla_t\eta\|_{L^2_{1,\delta}}^2+\|\eta\|_{L^2_{1,\delta}}^2+\|\hat\eta|_{t=\delta}\|_{L^2(\R)}^2+\|\eta_{2}|_{t=1}\|_{L^2(\R)}^2).
           \end{aligned}
       \end{equation}
    By replacing $\eta$ with $\nabla_s\eta$ in the above inequality, the following is true.
    \begin{align*}
        \|\nabla\nabla_s\eta\|_{L^2_{1,\delta}}^2\le&C(\|\nabla_s(\nabla_s\eta+J\nabla_t\eta)\|_{L^2_{1,\delta}}^2+\|\nabla\eta\|_{L^2_{1,\delta}}^2+|\Omega_2(\nabla_s\eta_2|_{t=1})|+|\Omega_{211}(\nabla_s\hat\eta|_{t=\delta})|)\\           \le&C(\|\nabla_s\eta+J\nabla_t\eta\|_{H^1_{1,\delta}}^2+\|\eta\|_{L^2_{1,\delta}}^2+\|\hat\eta|_{t=\delta}\|_{H^1(\R)}^2+\|\eta_{2}|_{t=1}\|_{H^1(\R)}^2).
    \end{align*}
    Therefore
    \begin{align*}
       \|\nabla\nabla_t\eta\|_{L^2_{1,\delta}}^2\le&C(\|\nabla\nabla_s\eta\|_{L^2_{1,\delta}}^2+\|\nabla_s\eta+J\nabla_t\eta\|_{H^1_{1,\delta}}^2+\|\eta\|_{L^2_{1,\delta}}^2+\|\hat\eta|_{t=\delta}\|_{H^1(\R)}^2+\|\eta_{2}|_{t=1}\|_{H^1(\R)}^2)\\
       \le&C(\|\nabla_s\eta+J\nabla_t\eta\|_{H^1_{1,\delta}}^2+\|\eta\|_{L^2_{1,\delta}}^2+\|\hat\eta|_{t=\delta}\|_{H^1(\R)}^2+\|\eta_{2}|_{t=1}\|_{H^1(\R)}^2)
    \end{align*}
    To finish the proof of the first inequality of $($a$)$, notice that
    \begin{align*}
        &\|\nabla_s\eta+J\nabla_t\eta\|_{H^1_{1,\delta}}\\
        \le&\|D^\delta\eta\|_{H^1_{1,\delta}}+\|\nabla_{\eta_2}J_2(u_2)\partial_tu_2\|_{H^1(\R\times[0,1])}+\|\nabla_{\hat\eta}\hat J(u)J(u)\partial_su\|_{H^1(\R\times[0,\delta])}\\
        \le&C(\|D^\delta\eta\|_{H^1_{1,\delta}}+\|\eta\|_{H^1_{1,\delta}}).
    \end{align*}
    Therefore it remains to show that
    \begin{equation}\label{ineq:5}
        \|\eta\|_{H^1_{1,\delta}}\le C(\|D^\delta\eta\|_{H^1_{1,\delta}}+\|\eta\|_{H^0_{1,\delta}}+\|(\eta_2',\eta_1,\eta_1')|_{t=\delta}\|_{H^1(\R)}+\|\eta_2|_{t=1}\|_{H^1(\R)}).
    \end{equation}
    Since
    \begin{align*}
        &\|\nabla\eta\|_{L^2_{1,\delta}}^2\\
        \le&C(\|\nabla_s\eta+J\nabla_t\eta\|_{L^2_{1,\delta}}^2+\|\eta\|_{L^2_{1,\delta}}^2+\|\hat\eta|_{t=\delta}\|_{L^2(\R)}^2+\|\eta_{2}|_{t=1}\|_{L^2(\R)}^2)\quad by\ \ (\ref{ineq:6})\\
        \le&C( \|D^\delta\xi\|_{L^2_{1,\delta}}^2+\|\nabla_{\xi_2}J_2(u_2)\partial_tu_2\|_{L^2}^2+\|\frac12\nabla_{\hat\xi}\hat J(\hat u)\hat J(u)\partial_t(u)\|_{L^2}^2\\
        &+\|\eta\|_{L^2_{1,\delta}}^2+\|\hat\eta|_{t=\delta}\|_{L^2(\R)}^2+\|\eta_{2}|_{t=1}\|_{L^2(\R)}^2)\\
        \le&C(\|D^\delta\eta\|_{H^1_{1,\delta}}+\|\eta\|_{H^0_{1,\delta}}+\|(\eta_2',\eta_1,\eta_1')|_{t=\delta}\|_{H^1(\R)}+\|\eta_2|_{t=1}\|_{H^1(\R)}).
    \end{align*}
    Thus Equation $($\ref{ineq:5}$)$ holds and this proves the first inequality of $($a$)$.\\
    To get the second inequality of $($a$)$, we apply Sobolev inequality (Theorem 5.4 Case B \cite{adams2003sobolev})
    \begin{align*}
        L^4(\R\times[0,1])\hookrightarrow H^1(\R\times[0,1])\\
        L^4(\R\times[0,\delta])\hookrightarrow H^1(\R\times[0,\delta])
    \end{align*}
    to $\|\nabla\eta\|_{L^4_{1,\delta}}$, so
    \[
    \|\nabla\eta\|_{L^4_{1,\delta}}\le C\|\nabla\eta\|_{H^1_{1,\delta}}\le C\|\eta\|_{H^2_{1,\delta}}.
    \]
    Therefore, combine this with the first inequality to get
    \begin{equation*}
            \begin{aligned}
            c_1\|\nabla\eta\|_{L^4_{1,\delta}}&\le\|D^\delta\eta\|_{H^1_{1,\delta}}+\|D^\delta\eta\|_{L^4_{1,\delta}}+\|\eta\|_{H^0_{1,\delta}}+\|(\eta_2',\eta_1,\eta_1')\|_{H^1(\R)}\\
            &+\|\eta_2|_{t=1}\|_{H^1(\R)},
        \end{aligned}
    \end{equation*}

    For the proof of $($b$)$, we argue as contradiction to prove a stronger version of the inequality
    \begin{align*}
            &\|\hat\eta|_{t=\delta}\|_{H^1(\R)}+\|\eta_2|_{t=1}\|_{H^1(\R)}+\|\eta\|_{H^0_{1,\delta}}+\|\eta_{2}\|_{H^1(\R\times[0,1])}\\
            \le &C(\|D^*_{u_2}\eta_2\|_{H^1(\R\times[0,1])}+\sqrt{\delta}\|\nabla_t\hat\eta\|_{H^1(\R\times[0,\delta])}).
    \end{align*}
    The proof of the first and the second one are similar, so the proof of the first one is presented here only. Suppose on the contrary, there are sequences $\{\eta^\nu=(\eta_2^\nu, \hat\eta^\nu)\}_\nu\subset H^2_{1,\delta^\nu}$ and $\{\delta^\nu\}_\nu\subset\Z_{>0}$ such that 
    \[
    \|\hat\eta^\nu|_{t=\delta}\|_{H^1(\R)}+\|\eta_2^\nu|_{t=1}\|_{H^1(\R)}+\|\eta^\nu\|_{H^0_{1,\delta^\nu}}+\|\eta_2^\nu\|_{H^1(\R\times[1,2])}=1, \quad \forall \nu
    \]
    but
    \[    \|D_{u_2}^*\eta_2^\nu\|_{H^1(\R\times[0,1])}+\sqrt{\delta}\|\nabla_t\hat\eta^\nu\|_{H^1(\R\times[0,\delta^\nu])}\rightarrow0,\quad \nu\rightarrow\infty.
    \]
    First note that
    \[
    \|\hat\eta^\nu|_{t=0}-\hat\eta^\nu|_{t=\delta^\nu}\|_{H^1(\R)}\le\int_0^{\delta^\nu}\|\nabla_t\hat\eta^\nu\|_{H^1(\R)}\le\sqrt{\delta^\nu}\|\nabla_t\hat\eta^\nu\|_{H^1(\R\times[0,\delta^\nu])}\rightarrow0.
    \]
    We then prove
    \[
    \|\eta^\nu_2|_{t=0}\|_{L^2(\R)}\le C,\quad\|\eta^\nu_2|_{t=0}\|_{L^2([-T,T])}\rightarrow0,\quad \forall T>0.
    \]
    The idea is the show that the convergent subsequences of $\{\eta^\nu\}$ is in $\ker(D^*_{u_2}\oplus\pi^\perp_2)$. Because $u_2$ is regular, $\ker(D^*_{u_2}\oplus\pi^\perp_2)=0$. These combine to get that the result needed. Applying Lemma $\ref{Lem:3.1.3}$ $($c$)$ to the second inequality below,
    \begin{align*}
        &\|\pi_2^\perp\eta_2^\nu|_{t=0}\|_{H^1(\R)}\\
        \le&\|\pi^\perp_2\eta'^\nu_2|_{t=\delta^\nu}\|_{H^1(\R)}+\|\eta'^\nu_2|_{t=0}-\eta'^\nu_2|_{t=\delta^\nu}\|_{H^1(\R)}+\|\eta'^\nu_2|_{t=0}-\eta^\nu_2|_{t=0}\|_{H^1(\R)}\\
        \le&C(\|\pi^\perp_{211}\hat\eta^\nu_2|_{t=\delta^\nu}\|_{H^1(\R)}+\|\hat\eta^\nu_2|_{t=0}-\hat\eta^\nu_2|_{t=\delta^\nu}\|_{H^1(\R)}+\|\eta'^\nu_2|_{t=0}-\eta^\nu_2|_{t=0}\|_{H^1(\R)}\\
        &+\|\eta'^\nu_1|_{t=0}-\eta^\nu_1|_{t=0}\|_{H^1(\R)}+\||\partial_s\hat u|_{t=0}||\hat\eta^\nu|_{t=\delta^\nu}|\|_{L^2(\R)}).
    \end{align*}
    Notice that the first, second, and forth terms above vanish because $\eta\in\Gamma_{1,\delta}$. Also, $\|\hat\eta^\nu_2|_{t=0}-\hat\eta^\nu_2|_{t=\delta^\nu}\|_{H^1(\R)}\rightarrow0$. So to prove $\|\pi_2^\perp\eta_2^\nu|_{t=0}\|_{H^1(\R)}\rightarrow0$, it remains to show that $\||\partial_s\hat u|_{t=0}||\hat\eta^\nu|_{t=\delta^\nu}|\|_{L^2(\R)}\rightarrow0$. According to Lemma \ref{Lem:3.1.3}, for all $T\in(0,\infty]$,
    \begin{align*}
        &\|\hat\eta^\nu|_{t=\delta^\nu}\|_{L^2([-T,T])}\\
        \le&C(\|\pi_2\eta'^\nu_2|_{t=\delta^\nu}\|_{L^2([-T,T])}+\|\pi_{211}^\perp\hat\eta^\nu_2|_{t=\delta^\nu}\|_{L^2([-T,T])}+\|(\eta'^\nu_1-\eta^\nu_1)|_{t=\delta^\nu}\|_{L^2([-T,T])})\\
        \le&C(\|\eta^\nu_2|_{t=\delta^\nu}\|_{L^2([-T,T])}+\|(\eta'^\nu_2-\eta^\nu_2)|_{t=0}\|_{L^2([-T,T])}+\|(\hat\eta^\nu|_{t=0}-\hat\eta^\nu|_{t=\delta^\nu}\|_{L^2([-T,T])}\\
        &+\|\pi_{211}^\perp\hat\eta^\nu_2|_{t=\delta^\nu}\|_{L^2([-T,T])}+\|(\eta'^\nu_1-\eta^\nu_1)|_{t=\delta^\nu}\|_{L^2([-T,T])}).
    \end{align*}
    Therefore $\|\hat\eta^\nu|_{t=\delta^\nu}\|_{L^2([-T,T])}\rightarrow0$. Combine this with the fact that
    \[
    \sup_{|s|\ge T}|\partial_s\hat u(s,0)|\rightarrow0\quad as\quad T\rightarrow\infty,
    \]
    we have
    \[
    \||\partial_s\hat u|_{t=0}||\hat\eta^\nu|_{t=\delta^\nu}|\|_{L^2(\R)}\rightarrow0.
    \]
    Therefore
    \[
    \|\pi_2^\perp\eta^\nu_2|_{t=0}\|_{H^1(\R)}\rightarrow0.
    \]
    The next step is to show that $\|\eta^\nu_2\|_{H^\frac32(\R\times[0,1])}\rightarrow0$.
    By the Fredholm theory of $D_{u_2}^*\oplus\pi^\perp_2$,
    \[
    \|\eta^\nu_2\|_{H^\frac32(\R\times[0,1])}\le C(\|D_{u_2}^*\eta^\nu_2\|_{H^1(\R\times[0,1])}+\|\pi^\perp_2\eta^\nu_2|_{t=0}\|_{H^1(\R)}+\|\eta^\nu_2\|_{H^0(\R\times[0,1])}).
    \]
    So we need to estimate the last term above.\\
    Define a cutoff function $h\in C_0^\infty(\R,[0,1])$ such that $h|_{\{|s|\leq T-1\}}\equiv 0$ and $h|_{\{|s|\geq T\}}\equiv 1$. Here we fix $T>1$ sufficiently large such that $u_{2}|_{supp(h)}=e_{x^\pm}(\vartheta_{2})$ for some smooth map $\vartheta_{2}:\{\pm s\geq (T-1)\} \to T_{x^\pm}M_{2}$. Then
    \begin{align*}
        \|\eta^\nu_2\|_{H^0(\R\times[0,1])}\le& \|h\eta^\nu_2\|_{H^0(\R\times[0,1])}+\|(1-h)\eta^\nu_2\|_{H^0(\R\times[0,1])}\\
        \le&C(\|h\eta^\nu_2\|_{H^0(\R\times[0,1])}+\|\eta^\nu_2\|_{H^0([-T,T]\times[0,1])}).
    \end{align*}
    To estimate $\|h\eta^\nu_2\|_{H^0(\R\times[0,1])}$, consider the operator $D_{x^\pm}=\partial_s-J(x^\pm)\partial_t$. Then \cite[Lemma~3.9, Proposition~3.14]{robbin1995spectral} implies the Fredholm property and bijectivity of $D_{x^\pm}$, as well as the following inequality
    \[
    \| \xi \|_{H^1(\R\times[0,1])}\leq C \Vert D_{x^\pm} \xi \Vert_{H^0(\R\times[0,1])}.
    \]
    In order to apply this estimate to $\eta_{2}^\nu$ we first find an extension $\zeta\in H^1(\R\times[0,1])$ of $\zeta|_{t=0}=\pi_{2}^\perp\eta^\nu_{2}|_{t=0}$ such that $\|\zeta\|_{H^1}\leq C\|\pi_{2}^\perp\eta^\nu_{2}|_{t=0}\|_{H^{1/2}}$.

Then we can apply the estimate to 
$\eta:=\Phi_{x^\pm}(\vartheta_{2})^{-1}\bigl( h(\eta^\nu_{2}-\zeta)\bigr)$, 
where $\Phi_{x^\pm}(\vartheta_{2})$ denotes parallel transport 
along the path 
$[0,1]\ni\tau\mapsto e_{x^\pm}(\tau\vartheta_{2})$.
We obtain
\begin{align*}
&\Vert h \eta^\nu_{2} \Vert_{H^1(\R\times[0,1])}  \\
\leq& C (\Vert \eta \Vert_{H^1(\R\times[0,1])} 
+ \Vert h \zeta \Vert_{H^1(\R\times[0,1])}) \\
\leq& C \bigl( 
\Vert \bigl( D_{x^\pm} - D_{u_2}^*\circ\Phi_{x^\pm}(\vartheta_{2}) \bigr) 
\eta \Vert_{H^0(\R\times[0,1])} 
+\Vert D_{u_2}^* (h\eta^\nu_{2}) \Vert_{H^0(\R\times[0,1])} 
+ \Vert h \zeta \Vert_{H^1(\R\times[0,1])} \bigr) \\
\le&C(\Vert \bigl( D_{x^\pm} - D_{u_2}^*\circ\Phi_{x^\pm}(\vartheta_{2})\bigr)\bigr|_{\{|s|>T-1\}}\Vert \cdot
\| h( \eta^\nu_{2} - \zeta) \Vert_{H^1(\R\times[0,1])}+\Vert D_{u_2}^* (h\eta^\nu_{2}) \Vert_{H^0(\R\times[0,1])}\\
&+\Vert h \zeta \Vert_{H^1(\R\times[0,1])})\\
\leq &C \bigl( 
\Vert \bigl( D_{x^\pm} - D_{u_2}^*\circ\Phi_{x^\pm}(\vartheta_{2}) 
\bigr)\bigr|_{\{|s|>T-1\}}\Vert \cdot
\| h \eta^\nu_{2} \Vert_{H^1(\R\times[0,1])} 
+\Vert D_{u_2}^*\eta^\nu_{2} \Vert_{H^0(\R\times[0,1])} \\
&+\Vert h \zeta \Vert_{H^1(\R\times[0,1])} \bigr)\\
\leq &C \bigl( 
\Vert \bigl( D_{x^\pm} - D_{u_2}^*\circ\Phi_{x^\pm}(\vartheta_{2}) 
\bigr)\bigr|_{\{|s|>T-1\}}\Vert \cdot
\| h \eta^\nu_{2} \Vert_{H^1(\R\times[0,1])} 
+\Vert D_{u_2}^*\eta^\nu_{2} \Vert_{H^0(\R\times[0,1])} \\
&+\Vert \eta^\nu_{2} \Vert_{H^0([-T,T]\times[0,1])} 
+ \|\pi_{2}^\perp\eta^\nu_{2}|_{t=0}\|_{H^{1/2}(\R)} \bigr) .
\end{align*}
Here the difference of the operators goes to zero for $T\to\infty$
since $u_{2}|_{\{|s|\geq T-1\}}\to x^\pm$ with all derivatives.  Thus for sufficiently large $T>0$ 
\begin{align*}
    &\Vert h \eta^\nu_{2} \Vert_{H^1(\R\times[0,1])}  \\
    \leq & \frac12\| h \eta^\nu_{2} \Vert_{H^1(\R\times[0,1])}+C\bigl( \Vert D_{u_2}^*\eta^\nu_{2}\Vert_{H^0(\R\times[0,1])}+\Vert \eta^\nu_{2} \Vert_{H^0([-T,T]\times[0,1])}+\|\pi_{2}^\perp\eta^\nu_{2}|_{t=0}\|_{H^{1/2}(\R)} \bigr) .
\end{align*}
Therefore
\[
\Vert h \eta^\nu_{2} \Vert_{H^1(\R\times[0,1])}\leq C\bigl( \Vert D_{u_2}^*\eta^\nu_{2}\Vert_{H^0(\R\times[0,1])}+\Vert \eta^\nu_{2} \Vert_{H^0([-T,T]\times[0,1])}+\|\pi_{2}^\perp\eta^\nu_{2}|_{t=0}\|_{H^{1/2}(\R)} \bigr) .
\]
After all this 
\[
\|\eta^\nu_2\|_{H^\frac32(\R\times[0,1])}\le C(\|D_{u_2}^*\eta^\nu_2\|_{H^1(\R\times[0,1])}+\|\pi^\perp_2\eta^\nu_2|_{t=0}\|_{H^1(\R)}+\|\eta^\nu_2\|_{H^0([-T,T]\times[0,1])}).
\]
    According to the Sobolev trace theorem, since $\|\eta^\nu_2\|_{H^\frac32(\R\times[0,1])}\rightarrow0$,
    \[
    \|\eta^\nu_2|_{t=0}\|_{H^1(\R)}+\|\eta^\nu_2|_{t=1}\|_{H^1(\R)}\rightarrow0
    \]
    By Lemma \ref{Lem:3.1.3},
    \begin{align*}
        &\|\hat\eta^\nu|_{t=\delta^\nu}\|_{H^1(\R)}\\
        \le&C(\|\pi_2\eta'^\nu|_{t=\delta^\nu}\|_{H^1(\R)}+\|\pi^\perp_{211}\hat\eta^\nu_2|_{t=\delta^\nu}\|_{H^1(\R)}+\|(\eta'^\nu_1-\eta^\nu_1)|_{t=\delta^\nu}\|_{H^1(\R)})\\
        \le&C(\|\hat\eta^\nu|_{t=0}\|_{H^1(\R)}+\|(\eta'^\nu_2-\eta^\nu_2)|_{t=0}\|_{H^1(\R)}+\|\hat\eta^\nu|_{t=0}-\hat\eta^\nu|_{t=\delta^\nu}\|_{H^1(\R)}\\
        &+\|\pi^\perp_{211}\hat\eta^\nu|_{t=\delta^\nu}\|_{H^1(\R)}+\|(\eta'^\nu_1-\eta^\nu_1)|_{t=0}\|_{H^1(\R)})
    \end{align*}
    \end{proof}

According to the definition of $\|\cdot\|_{\Omega_{1,\delta}}$
\begin{align*}
    &(1+c_2^{-1})\|(D^\delta)^*\eta\|_{\Omega_{1,\delta}}\\
    \ge &\frac12\|(D^\delta)^*\eta\|_{H^1_{1,\delta}}+\frac12\|(D^\delta)^*\eta\|_{L^4_{1,\delta}}+c_2^{-1}\|(D^\delta)^*\eta\|_{H^1(\R\times[0,1]))}\\
    \ge &\frac12\|(D^\delta)^*\eta\|_{H^1_{1,\delta}}+\frac12\|(D^\delta)^*\eta\|_{L^4_{1,\delta}}+(\frac12+\frac12)(\|(\eta_2',\eta_1,\eta_1')\|_{H^1(\R)}+\|\eta_2|_{t=1}\|_{H^1(\R)}+\|\eta\|_{H^0_{1,\delta}})\\
    -&c_2^{-1}\sqrt{\delta}\|\nabla_t(\eta_2',\eta_1,\eta_1')\|_{H^1(\R\times[0,\delta])}\quad(\text{Apply}\,(\ref{ineq:3}))\\
    \ge &\frac12c_1\|\eta\|_{H^2_{1,\delta}}+\frac12c_1\|\nabla\eta\|_{L^4_{1,\delta}}-c_2^{-1}\sqrt{\delta}\|\nabla_t(\eta_2',\eta_1,\eta_1')\|_{H^1(\R\times[0,\delta])}\quad(\text{Apply}\,(\ref{ineq:4}))\\
    =&\frac12c_1\|\eta\|_{\Gamma_{1,\delta}}-c_2^{-1}\sqrt{\delta}\|\nabla_t(\eta_2',\eta_1,\eta_1')\|_{H^1(\R\times[0,\delta])}\quad(\text{Let}\,\delta<\frac{c_1^2c_2^2}{16})\\
    \ge &\frac14c_1\|\eta\|_{\Gamma_{1,\delta}}.
\end{align*}
Similar estimates shows that for all $\eta\in\Gamma_{1,\delta}\cap K_0$
\begin{equation}\label{ineq:bddinv}
    \|D^\delta\eta\|_{\Omega_{1,\delta}}\ge\frac{c_1c_2}{4(c_2+1)}\|\eta\|_{\Gamma_{1,\delta}}.
\end{equation}
The above inequality
\[
(1+c_2^{-1})\|(D^\delta)^*\eta\|_{\Omega_{1,\delta}}\ge \frac14c_1\|\eta\|_{\Gamma_{1,\delta}}
\]
implies that $(D^\delta)^*$ is injective and hence $D^\delta$ is surjective. \textbf{It remains to show that the right inverse of $D^\delta$ is bounded}. If this is true, after applying McDuff and Salamon's theorem (Theorem (\ref{thm:mcdu})), we get the holomorphic representation needed.

\begin{cla}
    The right inverse of $D^\delta$ is bounded.
\end{cla}

\begin{proof}
    First notice that by the definition of $\Gamma_{1,\delta}$, $K_0\subset\Gamma_{1,\delta}$. The key point is that $D^\delta|_{K_0}$ is surjective. To prove this, we need to calculate the index of $D^\delta$. Using the correspondence given by Equations (\ref{equ:strcor1}) and (\ref{equ:strcor2}), as well as the relation between the index of Cauchy-Riemann operator and the Maslov index, we only need to count the Maslov index of $(u_1,u_2)$. Here $u_1$ is the $F_1$ component of the lift of $u_2|_{t=0}$ to $L_{12}\subset F_1\times F_2$, then apply the constant extension along the $[0,\delta]$ direction. Therefore the Maslov index of $(u_1,u_2)$ is
    \[
    \mathrm{Mas}(u_1,u_2)=\mathrm{Mas}(u_1)+\mathrm{Mas}(u_2)=\mathrm{Mas}(u_2)=1.
    \]
    As a result
    \[
    \mathrm{Ind}D^\delta=1.
    \]
    Combine this with the fact that $D^\delta$ is surjective, one has 
    \[
    \ker D^\delta=1.
    \]
    Because $D^\delta(K_0)\supset D^\delta(\ker (D^\delta)^\perp)$, then $D^\delta|_{K_0}$ is surjective. So $D^\delta$ has a right inverse
    \[
    (D^\delta)^{-1}:\Omega_{1,\delta}\rightarrow K_0\subset\Gamma_{1,\delta}.
    \]
   This is a bounded operator by Equation (\ref{ineq:bddinv}).
\end{proof}
\noindent Thus this ends the proof of the main theorem.
\end{proof}

\section{A potential example of figure-eight bubbling}\label{sec:5}
In \cite{wehrheim2009floer} Wehrheim and Woodward proved that for embedded monotone Lagrangians with restrictions on the Maslov index, then
\[
\mathrm{HF}(L_1,F,L_2;F_1,F_2)\cong\mathrm{HF}(L_1\circ F,L_2;F_2).
\]
In general this isomorphism does not hold because there are figure-eight bubblings when performing the strip shrinking method. Wehrheim and Woodward conjectured that the figure-eight bubblings obstruct the isomorphism. Bottman and Wehrheim \cite{bottman2018gromov} conjectured that these bounding cochains define the twisted Lagrangian Floer homology and produce the isomorphism. The purpose of this section is to give a potential example for Bottman and Wehrheim's conjecture.

\begin{ex}
    The Lagrangian correspondence $(g_1,g_2):F\looparrowright F_1\times F_2$ is given in Figure \ref{fig.exp}. The horizontal map is the 0-Dehn twist along the middle circles. The both vertical maps are foldings along the middle black circles. The left vertical map is $g_1$, and $g_2$ is the composition of the horizontal and the right vertical map. There is not any bigon in the left bottom. The boundary map $\mu_1$ of the Lagrangian Floer complex for the right bottom picture is given by counting the number of bigons shown in the following picture. 
    \begin{figure}[H]
    \centering 
    \includegraphics[width=0.5\textwidth]{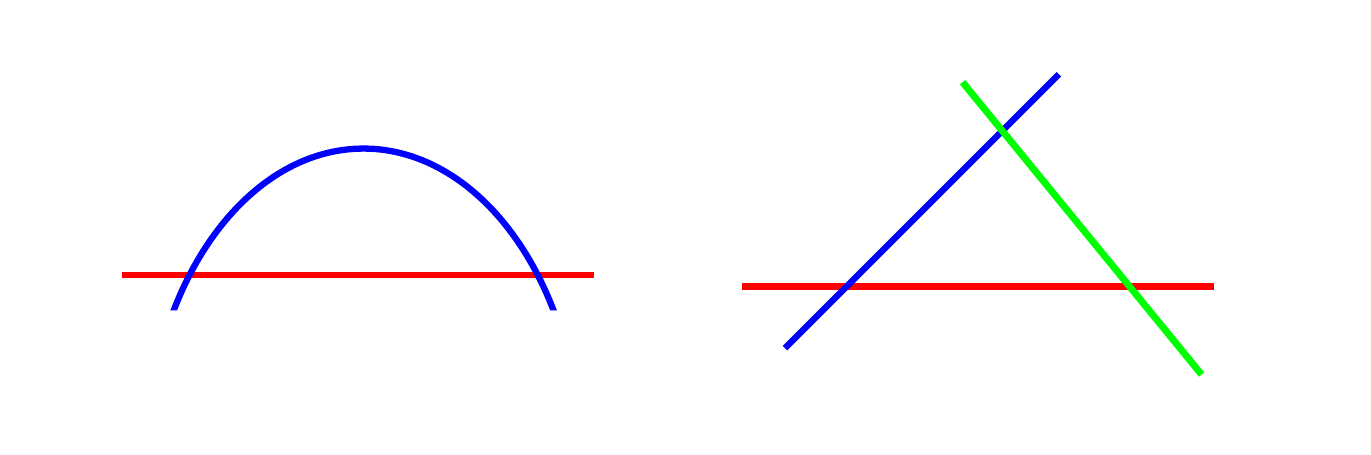}
    \caption{The boundary map $\mu_1$ is given by counting bigon as on the left and the right one corresponds to $\mu_2$.}
    \end{figure}    
    \noindent So there is a bigon in the right bottom picture. According to the main theorem of this paper, there is a holomorphic strip connecting two generators of $\CF^Q(L_1,F,L_2;F_1,F_2)$. This shows that there must be bubblings corresponding when shrinking the strip between $L_1$ and $F$ by \cite{bottman2018gromov}. 
    
    The generator $b$ in Figure \ref{fig.exp} is the bounding cochain. In fact, $\mu_i(\ldots,b,b,b,\ldots)=0, i\ge2$ and $\mu_0$ is defined as $b$. Therefore $\sum_{i\ge0}\mu_i(\ldots,b,\ldots)=\mu_0+\mu_1(b)=2b=0$ in $\Z_2$ coefficient. The twisted boundary map is defined as $\mu_1^b(\cdot):=\sum_{i\ge1}\mu_i({\ldots,b,\cdot,b,\ldots})$. One can show $(\mu_1^b)^2=0$ and $\mu_1^b(\cdot)=\mu_1(\cdot)+\mu_2(b,\cdot)+\mu_2(\cdot,b)$. By counting bigons and triangles in the large torus as blow, we have $\mu_1^b(\cdot)=0$. As a result
    \[
    \mathrm{HF}^b(L_1\circ F,L_2; F_2)\cong\mathrm{HF}(L_1,F\circ L_2;F_1)\cong
    Z_2\oplus\Z_2.
    \]

\begin{figure}[H]
\centering 
\includegraphics[width=1\textwidth]{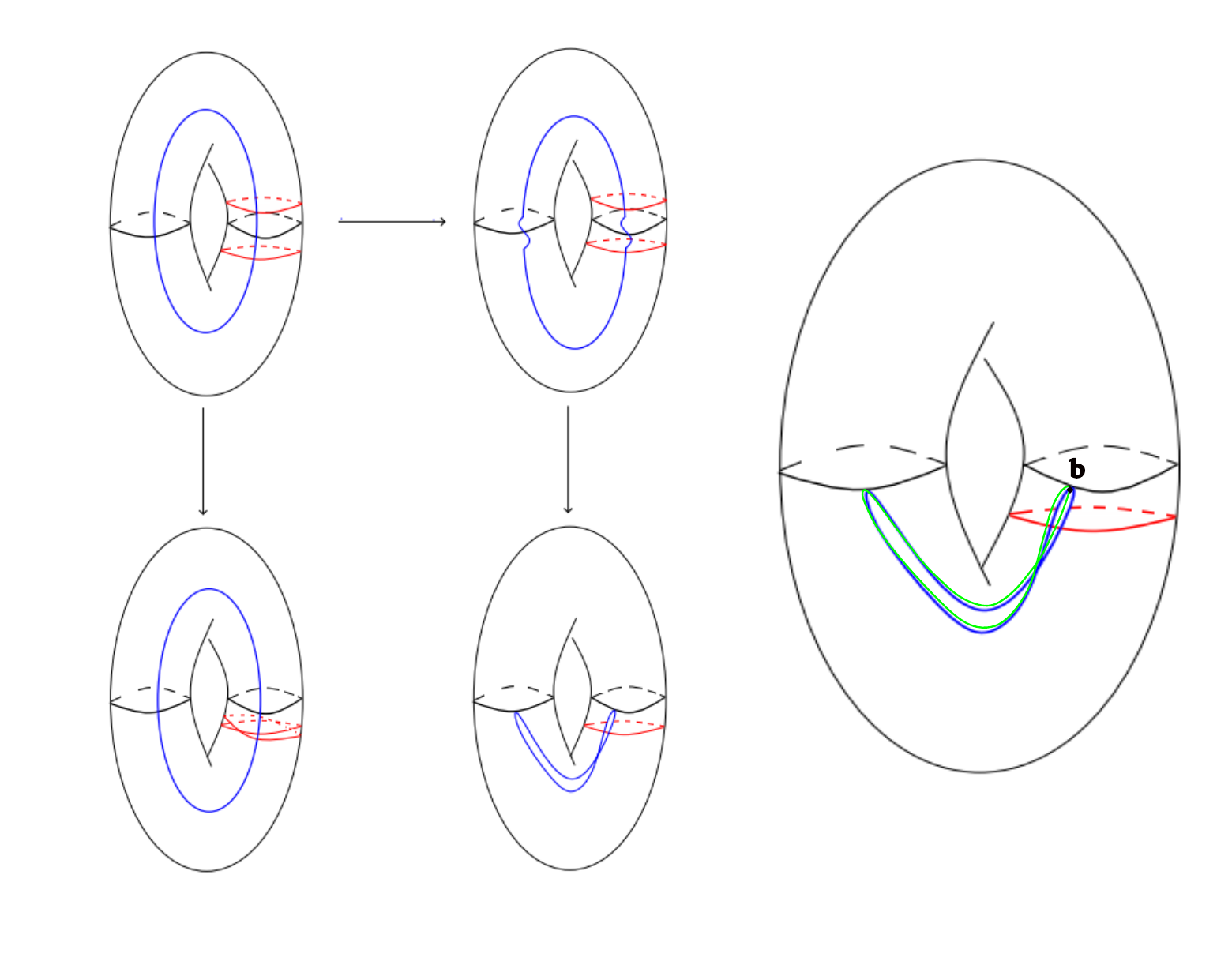}
\caption{The red curve is $L_1$ and the blue one is $L_2$. The green curve is the Hamiltonian perturbation of the blue curve.}
\label{fig.exp}
\end{figure}
\end{ex}

\begin{rmk}
    The reason that the above is only a potential example because the correspondence between bounding cochains and figure-eight bubblings must be functorial. This means that to solve Bottman and Wehrheim’s conjecture one has to figure out an universal way to construct the bounding cochains.
\end{rmk}

\bibliographystyle{plain}
\bibliography{references}

mailto:zhangzuyi1993@hotmail.com

\end{document}